\documentclass[conference]{IEEEtran}
\usepackage{graphics} 
\usepackage{epsfig} 
\usepackage{mathptmx} 
\usepackage{times} 
\usepackage{amsmath} 
\usepackage{amssymb} 
\usepackage{url}
\usepackage{comment}
\usepackage{color}

\usepackage{calrsfs}
\DeclareMathAlphabet{\pazocal}{OMS}{zplm}{m}{n}
 
\usepackage{algorithm}
\usepackage[noend]{algorithmic}

\newtheorem{theorem}{Theorem}

\newtheorem{lemma}{Lemma}

\newtheorem{claim}{Claim}

\newtheorem{example}{Example}

\newcommand{\R}{{\mathbb{R}}}

\newcommand{\E}{\mathbb{E}}

 \IEEEoverridecommandlockouts 

\begin{document}
\title{Secure State Estimation: Optimal Guarantees against Sensor Attacks in the Presence of Noise}
\author{\IEEEauthorblockN{Shaunak Mishra\IEEEauthorrefmark{1},
Yasser Shoukry\IEEEauthorrefmark{1},
Nikhil Karamchandani\IEEEauthorrefmark{1},
Suhas Diggavi\IEEEauthorrefmark{1} and
Paulo Tabuada\IEEEauthorrefmark{1}}
\IEEEauthorblockA{\IEEEauthorrefmark{1}Electrical Engineering Department, University of California, Los Angeles}
\thanks{The work was supported by NSF grant 1136174 and
DARPA under agreement number FA8750-12-2-0247.}
}

\maketitle

\begin{abstract}
Motivated by the need to secure cyber-physical systems against attacks,
we consider the problem of estimating the state of a noisy linear dynamical system
when a subset of sensors is arbitrarily corrupted by an adversary.
We propose a secure state estimation algorithm and derive (optimal) bounds on the achievable state estimation error.
In addition, as a result of independent interest, we give a coding theoretic interpretation for prior work on secure state estimation against sensor attacks in a noiseless dynamical system.
\end{abstract}

\IEEEpeerreviewmaketitle

\section{Introduction}
Cyber-physical systems (CPS) manage the vast majority of today's critical infrastructure
and securing such CPS against malicious attacks is a problem of growing importance \cite{Bullo_CSM}.
As a stepping stone towards securing complex CPS deployed in practice, several recent works have studied security problems in the context of linear dynamical systems \cite{Bullo_CSM, Hamza_TAC, YasserETPGarXiv, shaunak_CDC14, Yasser_SMT, Joao_ACC} leading to a fundamental understanding of how the system dynamics can be leveraged for security guarantees.
With this motivation, in this paper we focus on securely estimating the state of a linear dynamical system
from a set of noisy and maliciously corrupted sensor measurements.
We restrict the sensor attacks to be sparse in nature, \emph{i.e.}, an adversary can arbitrarily corrupt a subset of sensors in the system.

\par Prior work related to secure state estimation against sensor attacks in linear dynamical systems can be broadly categorized into three classes depending on the noise model for sensor measurements:
 1) noiseless 2) bounded non-stochastic noise, and 3) Gaussian noise.
For the noiseless setting, the work reported in \cite{Bullo_CSM, Hamza_TAC,YasserETPGarXiv} shows that, under a strong notion of observability, sensor attacks  (modeled as a sparse attack vector) can always be detected and isolated, and hence the state of the system can be exactly estimated.
In contrast, when the sensor measurements are affected by noise as well as maliciously corrupted, the problem of distinguishing between noise and attack vector arises.
Results reported in \cite{Yasser_SMT,Joao_ACC,Pajic_ICCPS} are representative of the second class: bounded non-stochastic noise. They provide sufficient conditions for distinguishing the sparse attack vector from bounded noise but do not guarantee the optimality of their estimation algorithm.
The work reported in this paper falls in the third class: Gaussian noise.
Prior work in this class includes \cite{YilinAllerton,Bai_Gupta,BoydKF, Georgios_TSP2}.
In \cite{YilinAllerton}, the analysis is restricted to detecting a class of sensor attacks called \textit{replay} attacks
(\emph{i.e.},
attacks in which legitimate sensor
outputs are replaced with outputs from previous time instants).
In \cite{Bai_Gupta}, the authors focus on the performance degradation of a scalar Kalman filter (\emph{i.e.}, scalar state and a single sensor) when the sensor is under attack.
Since they consider a single sensor setup,
attack sparsity across multiple sensors is not studied,
and in addition,
they focus on an adversary whose objective is to degrade the estimation performance and stay undetected at the same time
(thereby restricting the class of sensor attacks).
In \cite{BoydKF} and \cite{Georgios_TSP2},
robustification approaches for state estimation against sparse sensor attacks are proposed,
but they lack optimality guarantees against arbitrary sensor attacks.
\par In contrast to prior work in the Gaussian noise setup,
we consider a general linear dynamical system and give (optimal) guarantees on the achievable state estimation error against arbitrary sensor attacks.
The following toy example is illustrative of the nature of the problem addressed in this paper and some of the ideas behind our solution.

\begin{example}
Consider a linear dynamical system with a scalar state $x(t)$ such that $x(t+1) = x(t) + w(t)$, where $w(t)$ is the process noise following a Gaussian distribution with zero mean and is instantiated i.i.d. over time. The system has three sensors (indexed by $d$) with outputs $y_d(t) = x(t) + v_d(t)$, where $v_d(t)$ is the sensor noise at sensor $d$. Similarly to the process noise, $v_d(t)$ is Gaussian distributed with zero mean and is instantiated i.i.d. over time. The sensor noise is also independent across sensors.
Now, consider an adversary which can attack any one of the sensors and arbitrarily change its output.
In the absence of sensor noise, it is trivial to detect such an attack since the two good sensors (not attacked by the adversary) will have the same output.
Hence, a \textit{majority} based rule on the outputs leads to the exact state.
However, in the presence of sensor noise, even the good sensors may not have the same output and a simple majority based rule cannot be used for estimation.
In this paper, we build on the intuition that we may still be able to identify sensors whose outputs can lead to a \textit{good} state estimate by leveraging the noise statistics over a large enough time window.
In particular, our approach for this example would be to hypothesize a subset of two sensors as good, and then check whether the outputs from the two sensors are \textit{consistent} with the Kalman state estimate based on outputs from the same subset of sensors.
Furthermore, we show in this paper that such an approach leads to the optimal state estimation error for the given adversarial setup.
\end{example}

\par In this paper, we generalize the Kalman filter based approach in the above example to a general linear dynamical system with sensor and process noise.
Our main contributions can be listed as follows:
\begin{itemize}
\item We give optimal guarantees on the achievable state estimation error against arbitrary sensor attacks and propose an algorithm to achieve the same guarantees;
\item As a result of independent interest, we give a coding theoretic interpretation (alternate proof) for the necessary and sufficient conditions for secure state estimation in the absence of noise \cite{ Hamza_TAC, YasserETPGarXiv, Joao_ACC} (known as the sparse observability condition).
\end{itemize}

\par The remainder of this paper is organized as follows. Section~\ref{sec:setup} deals with the setup.
The main results are stated in Section~\ref{sec:main_results}.
Section~\ref{sec:scalar_state} considers the simpler setting of a scalar state and illustrates the main ideas behind our estimation algorithm and Section~\ref{sec:vector_state_prediction} considers its generalization to a vector state. Finally, we discuss the coding theoretic view of the sparse observability condition \cite{YasserETPGarXiv} in Section~\ref{sec:sparse_proof}.

\section{Setup} \label{sec:setup}
\subsection{System model} 
We consider a linear dynamical system with sensor attacks as shown below:
\begin{align}
\mathbf{x}\left(t+1\right) = \mathbf{Ax}(t) + \mathbf{w}(t), \quad \mathbf{y}(t) = \mathbf{C} \mathbf{x}(t) + \mathbf{v}(t) + \pmb{\phi}(t), \label{eq:vector_system_model}
\end{align} where $\mathbf{x}(t)\in \mathbb{R}^n$ denotes the state of the plant at time $t \in \mathbb{N}$,  $\mathbf{w}(t) \in \mathbb{R}^n$ denotes the process noise at time $t$, $\mathbf{y}(t) \in \mathbb{R}^p$ denotes the output of the plant at time $t$ and $ \mathbf{v}(t) \in \mathbb{R}^p$ denotes the sensor noise at time $t$. The process noise $\mathbf{w}(t) \sim \mathcal{N} \left ( \mathbf{0}, \sigma^2_w \mathbf{I}_n\right)$, \emph{i.e.}, $\mathbf{w}(t)$ is Gaussian distributed with zero mean and covariance matrix $\sigma^2_w \mathbf{I}_n$, where $ \mathbf{I}_n$ is the identity matrix of dimension $n$ and $\sigma_w \in \mathbb{R}$. Similarly, sensor noise $\mathbf{v}(t) \sim \mathcal{N} \left ( \mathbf{0}, \sigma^2_v \mathbf{I}_p\right)$. Both $\mathbf{v}(t)$ and $\mathbf{w}(t)$ are instantiated i.i.d. over time, and $\mathbf{v}(t)$ is independent of $\mathbf{w}(t)$.

\par The sensor attack vector $\pmb{\phi}(t) \in \R^p$ in \eqref{eq:vector_system_model} is introduced by a $k$-adversary defined as follows.
A $k$-adversary has access to any $k$ out of the $p$ sensors in the system.
Specifically, let $\pmb{\kappa} \subseteq \{1,2,\ldots p\}$ denote the set of attacked sensors
(with $|\pmb{\kappa}| = k$).
The $k$-adversary can observe the actual outputs in the $k$ attacked sensors and change them arbitrarily.
Specifically, the output of an attacked sensor $j \in \pmb{\kappa}$ can be expressed as
\begin{align}
y_j (t) = \mathbf{c}^T_j \mathbf{x}(t) + v_j(t) + \phi_j(t),
\end{align}
where $T$  denotes the matrix transpose operation,
$\mathbf{c}_j^T$ is the $j${th} row of $\mathbf{C}$,
$v_j(t)$ is the noise at sensor $j$
and $\phi_j(t)$
is the adversarial corruption introduced at sensor $j$.
For $j \notin \pmb{\kappa}$, $\phi_j(t) =0$.
The adversary's choice of $\pmb{\kappa}$ is unknown but is assumed to be constant over time (static adversary).
The adversary is assumed to have unbounded computational power, and knows the system parameters (\emph{e.g.,} $\mathbf{A}$ and $\mathbf{C}$) and noise statistics (\emph{e.g.,} $\sigma^2_w$ and $\sigma^2_v$). 
However, the adversary is limited to have only causal knowledge of the process noise and the sensor noise in good sensors (not attacked by the adversary). We discuss this assumption in more detail in Section~\ref{sec:causal_assumptions}.
\subsection{State estimation: prediction and filtering}
In this paper, we address two state estimation problems: (1) state prediction and (2) state filtering. 

\par In the state prediction problem, the goal is to estimate the state at time $t$ based on outputs till time $t-1$.
In the absence of sensor attacks,
using a Kalman filter for predicting the state in \eqref{eq:vector_system_model} leads to the optimal (MMSE) error covariance asymptotically \cite{kailath_book}.
In particular, 
the Kalman filter update rule can be written as:
\begin{align}
\hat{\mathbf{x}}(t+1) = \mathbf{A}\hat{\mathbf{x}} (t) + \mathbf{L}(t) \left(\mathbf{y}(t) - \mathbf{C} \hat{\mathbf{x}}(t) \right), \label{eq:kalman_prediction_update}
\end{align}
where $\hat{\mathbf{x}}(t+1)$ is the state estimate at time $t+1$ and $\mathbf{L}(t)$ is the Kalman filter gain.
For a Kalman filter in steady state \cite{kailath_book}, the steady state gain satisfies $\mathbf{L}(t) = \mathbf{L}$.
Also, we use $P_{opt,\mathbf{s}}$ to denote the trace of steady state (prediction) error covariance matrix \cite{kailath_book} obtained by using a Kalman filter on a sensor subset $\mathbf{s} \subseteq \{1,2,\ldots p\}$.

\par In contrast to the prediction problem, the goal in the state filtering problem is to estimate the state at time $t$ based on outputs till time $t$. In the absence of sensor attacks,
a Kalman filter update rule similar to \eqref{eq:kalman_prediction_update} can be used for the filtering problem \cite{kailath_book} (see Appendix \ref{sec:filtering} for details) and we use $F_{opt,\mathbf{s}}$ to denote the trace of steady state (filtering) error covariance matrix obtained by using a Kalman filter on a sensor subset $\mathbf{s}$.

\subsection{Causal knowledge assumptions} \label{sec:causal_assumptions}
At time $t$, the attack vector $\pmb{\phi}(t)$ in \eqref{eq:vector_system_model} depends on the knowledge of the adversary at time $t$, and in this context, we limit the adversary's knowledge of the process and sensor noise along the lines of causality.
In particular,
for the prediction problem we assume the following for a $k$-adversary:
\begin{itemize}
\item[(A1)] The adversary's knowledge at time $t$ is statistically independent of $\mathbf{w}(t')$ for $t' > t$, \emph{i.e.}, $\pmb{\phi}(t)$ is statistically independent of $\{\mathbf{w}(t')\}_{t'>t}$;
\item[(A2)] For a \textit{good} sensor $d \in \{1,2,\ldots p\} - \pmb{\kappa}$,
the adversary's knowledge at time $t$ (and hence $\pmb{\phi}(t)$) is statistically independent of $\{v_d(t') \}_{t' > t}$.
\end{itemize}
Intuitively, assumptions (A1) and (A2) limit the adversary to have only causal knowledge of the process noise and the sensor noise in good sensors (not attacked by the adversary).
Note that, apart from (A1) and (A2), we do not impose any restrictions on the statistical properties, boundedness and the time evolution of the corruptions introduced by the $k$-adversary.
In the filtering problem, we replace assumptions (A1) and (A2) with (A3) and (A4) as described below:
\begin{itemize}
\item[(A3)] The adversary's knowledge at time $t$ is statistically independent of $\mathbf{w}(t')$ for $t' \geq t$, \emph{i.e.}, $\pmb{\phi}(t)$ is statistically independent of $\{\mathbf{w}(t')\}_{t' \geq t}$;
\item[(A4)] For a good sensor $d \in \{1,2,\ldots p\} - \pmb{\kappa}$,
the adversary's knowledge at time $t$ (and hence $\pmb{\phi}(t)$) is statistically independent of
$\{v_d(t') \}_{t' \geq t}$.
\end{itemize}
Clearly, (A3) is a stronger version
of (A1), requiring $\pmb{\phi}(t)$ to be independent of $\mathbf{w}(t)$. Similarly, (A4) is a stronger version of (A2).

\subsection{Sparse observability condition}
For the matrix pair $\left( \mathbf{A}, \mathbf{C} \right)$, the observability matrix $\mathbf{O}$ with observability index $\mu$ is defined
as shown below:
\begin{align}
\mathbf{O} = 
\begin{bmatrix} 
\mathbf{C} \\ \mathbf{C} \mathbf{A} \\ \vdots \\ \mathbf{C} \mathbf{A}^{\mu-1} 
\end{bmatrix}.
\end{align}
In this context, a linear dynamical system, characterized by the pair $\left( \mathbf{A}, \mathbf{C} \right)$, is said to be observable if there exists a positive integer $\mu$ such that $\mathbf{O}$ has full column rank.
In the absence of sensor and process noise, the conditions under which state estimation can be done despite sensor attacks have been studied in \cite{Hamza_TAC,YasserETPGarXiv,Joao_ACC}.
In particular, a linear dynamical system as shown in \eqref{eq:vector_system_model} is called $\theta$-sparse observable if for every subset $\mathbf{s} \subseteq \{1,\ldots p\}$ of size $\theta$, the pair $\left( \mathbf{A} , \mathbf{C}_{\mathbf{s}}\right)$ is observable (where $\mathbf{C}_{\mathbf{s}}$ is formed by the rows of $\mathbf{C}$ corresponding to
sensors indexed by the elements of $\mathbf{s}$).
Also, $\theta$ is the smallest positive integer to satisfy the above
observability
property.
The condition:
\begin{align}
\theta \leq p - 2k, \label{eq:theta_condition}
\end{align}
is necessary and sufficient for \textit{exact} state estimation against a $k$-adversary in the absence of process and sensor noise \cite{YasserETPGarXiv}; we will refer to this condition as the sparse observability condition.
We provide a coding theoretic interpretation for the same in Section~\ref{sec:sparse_proof}.

\section{Main results} \label{sec:main_results}
We first state our achievability result followed by an impossibility result.
\begin{theorem}[\textbf{Achievability}] \label{thm:ach}
Consider the linear dynamical system defined in~\eqref{eq:vector_system_model} satisfying the sparse observability condition \eqref{eq:theta_condition} against a $k$-adversary.
Assuming (A1) and (A2),
and a time window $G = \{ t_1, t_1 +1 , \ldots t_1 +N-1\}$ for the state prediction problem,
the following bound on the prediction error is achievable against a $k$-adversary.
For any $\epsilon > 0$ and $\delta > 0$,
there exists a large enough $N$ such that:
\begin{align}
\mathbb{P} \left ( 
\frac{1}{N} \sum_{t \in G } \mathbf{e}^T(t)\mathbf{e}(t) \le \max_{\mathbf{s} \subset \{1,2,\ldots p\}, \; |\mathbf{s}| = p-k} \left(  P_{opt,\mathbf{s}}  \right)  + \epsilon 
\right ) \geq 1- \delta ,
 \label{eq:sse}
\end{align}
where $\mathbf{e}(t) = \mathbf{x}(t) -\hat{ \mathbf{x}}(t)$ is the estimation error for the state estimate $\hat{ \mathbf{x}}(t)$.
In other words, with high probability (w.h.p.), the bound
$\displaystyle{\limsup_{N\rightarrow \infty} 
\frac{1}{N} \sum_{t \in G } \mathbf{e}^T(t)\mathbf{e}(t) \le \max_{\mathbf{s} \subset \{1,2,\ldots p\}, \; |\mathbf{s}| = p-k} \left(  P_{opt,\mathbf{s}} \right)}$
is achievable.
Similarly, for the state filtering problem, assuming (A3) and (A4) against a \mbox{$k$-adversary}, the following bound on the corresponding filtering error $\mathbf{e}(t)$ is achievable w.h.p.:
\begin{align}
 \limsup_{N\rightarrow \infty} \frac{1}{N} \sum_{t \in G } \mathbf{e}^T(t)\mathbf{e}(t) \le \max_{\mathbf{s} \subset \{1,2,\ldots p\}, \; |\mathbf{s} |= p-k}  \left ( F_{opt,\mathbf{s}} \right) . \label{eq:ssf}
\end{align}
\end{theorem}
The achievability in Theorem~\ref{thm:ach} is through our proposed algorithms, which we discuss in the following sections. The impossibility result can be stated as follows.

\begin{theorem}[\textbf{Impossibility}] \label{thm:impos}
Consider the linear dynamical system defined in~\eqref{eq:vector_system_model} and an oracle MMSE estimator that has knowledge of $\pmb{\kappa}$, \emph{i.e.}, the set of sensors attacked by a $k$-adversary.
Then, there exists an attack sequence $\pmb{\phi}(t)$ such that the trace of the prediction error covariance of the oracle estimator is bounded from below as follows:
\begin{align}
 tr \left ( \E \left ( \mathbf{e}(t) \mathbf{e}^T(t) \right )  \right)  \ge  P_{opt,\mathbf{s}},
\end{align}
where $\mathbf{e}(t)$ above is the oracle estimator's prediction error and $\mathbf{s} = \{1,2,\ldots p\} - \pmb{\kappa}$.
Similarly, for the filtering problem,
\begin{align}
 tr \left ( \E \left (\mathbf{e}(t) \mathbf{e}^T(t) \right ) \right)  \ge  F_{opt,\mathbf{s}} .
\end{align}
\end{theorem}

\IEEEproof{
Consider the attack scenario where the outputs from all attacked sensors are equal to zero, \emph{i.e.}, the
corruption $\mathbf{\phi}_j(t) = - \mathbf{c}^T_j \mathbf{x}(t) - v_j(t), \; \forall j \in \pmb{\kappa}$.
Hence, the information collected from the attacked sensors cannot enhance the estimation performance. Accordingly, the estimation performance from the remaining sensors is the best 
one can expect to achieve.
}

\par Clearly, for the adversary's \textit{best} choice of $\pmb{\kappa}$, the guarantees given in our achievability match the impossibility bound (in an empirical average sense), and hence, we consider our guarantees \textit{optimal}. We measure the performance of our proposed algorithms in terms of empirical average (and not expectation) since the resultant error in the presence of attacks may not be ergodic.

\section{Secure state estimation: scalar state}
\label{sec:scalar_state}
In this section, we illustrate the main ideas behind our general scheme in the simpler setting of estimating a scalar state variable against a $k$-adversary. In particular, we focus on the state prediction problem for
the system in \eqref{eq:vector_system_model} when the state is a scalar and there are $p \geq  2k+1$ sensors (\emph{i.e.}, $1$-sparse observability condition against $k$-adversary).
For clarifying the presence of scalar terms in our analysis, we use the scalar version (regular instead of bold face) of the notation developed in Section~\ref{sec:setup}, \emph{i.e.}, $x(t)$ for the plant's state, $\hat{x}(t)$ for the estimate, and $y_d(t) = c_d x(t) + v_d(t)$ for the output of a good sensor \mbox{$d \in \{1,2,\ldots p\} - \pmb{\kappa}$}. We first describe our proposed algorithm
for a time window $G =\{t_1, t_1 + 1, \ldots t_1 + N-1 \}$ of size $N$, and then analyze its performance.

\paragraph*{Secure scalar state prediction algorithm}
Considering a time window $G$, Algorithm \ref{alg:scalar} shows the secure state prediction algorithm for the case when the state is a scalar. 
The algorithm runs a bank of $p\choose{p-k}$ Kalman filters in parallel;
one Kalman filter associated with each distinct set of $p-k$ sensors.
For each distinct set $\mathbf{s}$ of $p-k$ sensors, the corresponding Kalman filter fuses all the measurements from these sensors in order to calculate (prediction) estimate $\hat{x}_{\mathbf{s}}(t)$. Using the calculated estimate $\hat{x}_{\mathbf{s}}(t)$, we calculate the individual residues for each sensor as shown in~\eqref{eq:indiv_residual}.
The algorithm, then, exhaustively searches for the set $\mathbf{s}$ of $p-k$ sensors which satisfy the residue test shown in~\eqref{eq:residue_test_scalar}.
If a set $\mathbf{s}^\star$ satisfies the residue test, it is declared \textit{good} and the corresponding Kalman estimate $\hat{x}_{\mathbf{s}^\star}(t)$ is used as the state estimate for the given time window. Intuitively, the residue test checks if the outputs from a given sensor set $\mathbf{s}$ are \textit{consistent} with the corresponding Kalman estimate over the time window $G$.

\begin{algorithm}[t]
\caption{\textsc{Secure State Prediction - scalar case}}
\begin{algorithmic}[1]
\STATE Enumerate all sets $\mathbf{s} \in \mathbf{S} $ such that: \\ \mbox{$\qquad \qquad \mathbf{S} = \{ \mathbf{s} | \mathbf{s} \subset \{1,2, \ldots p\}, \; |\mathbf{s}| = p-k \}$}.
\STATE  For each $\mathbf{s} \in \mathbf{S}$, run a Kalman filter that uses all sensors indexed by $\mathbf{s}$ and returns estimate $\hat{x}_{\mathbf{s}}(t) \in \R$.
\STATE For each $\mathbf{s} \in \mathbf{S}$, calculate the residues for all sensors $d \in \mathbf{s}$ over a time window $G = \{t_1,t_1+1 , \ldots t_1 +N -1 \}$ as:
\begin{align}
r_d(t) = y_d(t) - c_d \hat{x}_{ \mathbf{s}}(t) \qquad \forall d \in \mathbf{s}, \quad \forall t \in G.
\label{eq:indiv_residual}
\end{align}

\STATE  Pick the set $\mathbf{s}^{\star} \in \mathbf{S}$ which satisfies the following residue test:
\begin{align}
\frac{1}{N}\sum_{t \in G} r_d^2(t) \leq c^2_d P_{opt,\mathbf{s}^\star} +\sigma_v^2 + \epsilon \qquad \forall d \in \mathbf{s}^{\star}, \label{eq:residue_test_scalar}
\end{align}
where $\epsilon \geq  0 $ is a design parameter and can be made arbitrarily small for large enough $N$.
\STATE Return $\mathbf{s}^{\star}$ and $\hat{x}(t) := \hat{x}_{\mathbf{s}^\star}(t) \quad \forall t  \in G$.
\end{algorithmic}
\label{alg:scalar}
\end{algorithm}

\paragraph*{Performance analysis}
Consider the set $\mathbf{s}$ of $p-k$ sensors which are not attacked by the $k$-adversary.
Assuming that the Kalman filter corresponding to set $\mathbf{s}$ is in steady state,
it can be shown that $\mathbb{E} \left( r_d^2(t) \right) =  c^2_d P_{opt,\mathbf{s}} +\sigma_v^2$, $\forall d \in \mathbf{s}$ \cite{kailath_book} (where residue $r_d(t)$ is as defined in \eqref{eq:indiv_residual}).
For large enough $N$, due to the (strong) law of large numbers (LLN), the residue test will be satisfied w.h.p. for at least this set of good sensors.
This ensures that w.h.p., the algorithm will not return an empty set. Also, the estimate $\hat{x}_{\mathbf{s}}(t)$ from this set of good sensors trivially achieves the error bound \eqref{eq:sse}.
But, since the algorithm can return any set of size $p-k$ which satisfies the residue test,
it may be possible that some of the sensors in the returned set are corrupt.
In the remainder of our analysis, we show that for \textit{any} set returned by the algorithm, the corresponding Kalman estimate achieves \eqref{eq:sse}.

\par Suppose the algorithm returns a set $\mathbf{s}$ of $p-k$ sensors. There is definitely one good sensor (say sensor $d$) in this set because there can be at most $k$ attacked sensors and $p - k > k$.
Since the residue test is satisfied for this sensor, we have the following constraint:
\allowdisplaybreaks[4]
{\begin{align}
\frac{1}{N} \sum_{t \in G} r_d^2(t) & \stackrel{(a)}= \frac{1}{N} \sum_{t \in G} \left ( c_d x(t) + v_d(t) - c_d \hat{x}_{\mathbf{s}}(t)\right )^2 \nonumber \\
& = \frac{1}{N} \sum_{t \in G} \left ( c_d e(t) + v_d (t) \right )^2 \nonumber\\
& = \frac{c_d^2}{N} \sum_{t \in G}  e^2 (t) + \frac{1}{N} \sum_{t \in G}  v^2_d(t) + \frac{2 c_d}{N} \sum_{t \in G}  e(t) v_d(t) \nonumber \\
& \stackrel{(b)} \leq c_d ^2 P_{opt,\mathbf{s}} +\sigma_v^2 + \epsilon, \label{eq:scalar_step1}
\end{align}}where (a) follows from $y_d(t) = c_d x(t) + v_d(t)$ for a good sensor $d$ and (b) follows from the residue test.
The error $e(t)$ above is the state estimation (prediction) error at time $t$ (in the presence of a $k$-adversary) when $\hat{x}_{\mathbf{s}}(t)$ is used as the state estimate.
Using LLN, we can make an additional simplification as follows.
For any $\epsilon > 0$, there exists a large enough $N$ such that:
\allowdisplaybreaks[4]
{
\begin{align}
& \frac{c_d^2}{N} \sum_{t \in G}  e^2 (t) + \frac{2 c_d}{N} \sum_{t \in G}  e(t) v_d(t) 
 \label{eq:cross_term_scalar} \\
& \stackrel{(a)} \leq c_d^2 P_{opt,\mathbf{s}} +  \left | \sigma_v^2 - \frac{1}{N} \sum_{t \in G} v^2_d(t) \right |  + \epsilon 
 \stackrel{(b)} \leq c_d^2 P_{opt,\mathbf{s}} + 2 \epsilon,   \label{eq:scalar_step2}
\end{align}}where (a) follows from \eqref{eq:scalar_step1},
and (b) follows w.h.p. due to LLN.
Our next step will be to show that the cross term $\frac{2 c_d}{N} \sum_{t \in G} e(t) v_d(t)$ in \eqref{eq:cross_term_scalar} is vanishingly small w.h.p. as $N \rightarrow \infty$;
this leads to the required bound on $\frac{1}{N} \sum_{t\in G} e^2 (t)$ using \eqref{eq:scalar_step2}.
We do so in two steps: first we show that the mean of the cross term $\frac{2 c_d}{N} \sum_{t\in G} e(t) v_d(t)$ is zero and then show that its variance is vanishingly small as $N\rightarrow \infty$.

\par The mean of the cross term $\frac{2 c_d}{N} \sum_{t \in G} e(t) v_d(t)$ can be computed as shown below:
\begin{align}
\mathbb{E} \left( \frac{2c_d}{N} \sum_{t \in G} e(t) v_d(t) \right) \stackrel{(a)} =\frac{2 c_d }{N} \sum_{t \in G} \mathbb{E} \left ( e(t) \right) \mathbb{E} \left( v_d(t) \right) = 0, \label{eq:zero_mean_scalar}
\end{align}
where (a) follows from the independence of $e(t)$ from $v_d(t)$ (due to assumption (A2), $\hat{x}_{\mathbf{s}}(t)$ is independent of good sensor noise $v_d(t)$ despite sensor attacks).
Also, using \eqref{eq:zero_mean_scalar} and taking the expectation in \eqref{eq:cross_term_scalar}:
\begin{align}
 \mathbb{E} \left ( \frac{1}{N} \sum_{t \in G} e^2(t) \right) 
& \leq P_{opt,\mathbf{s}} + \frac{2\epsilon}{c^2_d} .
\label{eq:first_order_bound_c_i}
\end{align}
As the final step in our analysis, we will now show that the variance of cross term $\frac{2 c_d}{N} \sum_{t \in G} e(t) v_d(t)$ is vanishingly small as $N \rightarrow \infty$. For any $\epsilon_1 > 0$, there exists a large enough $N$ such that:
\allowdisplaybreaks[4]
{
\begin{align}
&  \mathbb{E} \left(  \left (\frac{1}{N} \sum_{t \in G} e(t) v_d(t)  \right)^2  \right)  \nonumber \\
&= \frac{ \sum_{t \in G } \mathbb{E} \left(  e^2(t) v^2_d(t) \right)}{N^2}  + \frac{2}{N^2} \sum_{t,\;t' \in G, \; t < t' } \mathbb{E} \left(  e(t) v_d(t) e(t') v_d(t') \right) \nonumber \\
& \stackrel{(a)}= \frac{1}{N^2} \sum_{t \in G} \mathbb{E} \left( e^2(t) \right) \mathbb{E} \left( v^2_d(t) \right)
\nonumber\\
& \quad + \frac{2}{N^2} \sum_{t,\; t' \in G, \; t < t' } \mathbb{E} \left( e(t) v_d(t) e(t') \right) \mathbb{E} \left( v_d(t') \right)  \nonumber \\
& = \frac{\sigma_v^2}{N} \mathbb{E} \left( \frac{\sum_{t \in G } e^2(t)}{N} \right) \stackrel{(b)} \leq 
\epsilon_1 ,
\end{align}}where (a) follows from the independence of $e(t)$ from $v_d(t)$ and the independence of $v_d(t')$ from $e(t) v_d(t) e(t') $ (for $t' > t$), (b) follows from \eqref{eq:first_order_bound_c_i}.
The above result implies that the cross term $\frac{2 c_d}{N} \sum_{t \in G} e(t) v_d(t)$ (with zero mean) has vanishingly small variance as $N \rightarrow \infty$.
As a result, using Chebyshev's inequality and \eqref{eq:scalar_step2}, we have the error bound \eqref{eq:sse}.

\section{Secure state estimation: vector state}  \label{sec:vector_state_prediction}
In this section, we consider the state estimation problem (against a $k$-adversary) for the general linear dynamical system described in \eqref{eq:vector_system_model}, when the state is a vector.
We focus on the prediction problem in this section; the filtering problem is studied in Appendix~\ref{sec:filtering}.
We assume that the system is \mbox{$\theta$-sparse} observable such that it satisfies the sparse observability condition \eqref{eq:theta_condition} against a $k$-adversary. We first introduce some additional notation required for our proposed algorithm.

\paragraph*{Additional notation} Consider a set $\mathbf{s}$ of $p-k$ sensors. Such a set has ${p-k}\choose{\theta}$ sensor subsets of size $\theta$, and we index these subsets of $\mathbf{s}$ by $i$. Due to the $\theta$-sparse observability condition, each subset $i$ forms an observable pair $\left ( \mathbf{A} , \mathbf{C}_i\right)$ with observability matrix $\mathbf{O}_i$ and observability index $\mu_i$;
$\mathbf{C}_i$ is formed by rows of $\mathbf{C}$ corresponding to subset $i$ of $\mathbf{s}$.
We define matrices $\mathbf{J}_i$ and $\mathbf{M}_{\mu_i}$ as shown below:
\begin{align}
\mathbf{J}_i & = \begin{bmatrix} \mathbf{0} & \mathbf{0}  & \hdots & \mathbf{0} \\
\mathbf{C}_i & \mathbf{0}  & \hdots & \mathbf{0} \\
\mathbf{C}_i \mathbf{A} & \mathbf{C}_i & \hdots & \mathbf{0} \\
\vdots & \vdots & \ddots & \vdots \\
\mathbf{C}_i \mathbf{A}^{{\mu_i}-2} & \mathbf{C}_i \mathbf{A}^{{\mu_i}-3 } & \hdots & \mathbf{C}_i \\
\end{bmatrix}, 
\; \mathbf{M}_{\mu_i} =  \sigma^2_w \mathbf{J}_i \mathbf{J}^T_i + \sigma^2_v \mathbf{I}_{\mu_i}.
\end{align}
The pseudo-inverse of $\mathbf{O}_i$ is denoted by $\mathbf{O}_i^{\dagger}$.
The output from sensor subset $i$ (of size $\theta$) at time $t$ is denoted by $\mathbf{y}_i(t) \in \mathbb{R}^{\theta}$.
We consider the state estimation problem for a time window $G$ of size $N$ and assume without loss of generality that $\mu_i$ divides $N$ such that $\mu_i N_B=  N$. 

\paragraph*{Secure state prediction algorithm}

\begin{algorithm}[t]
\caption{\textsc{Secure State Prediction - vector case}}
\begin{algorithmic}[1]
\STATE Enumerate all sets $\mathbf{s} \in \mathbf{S} $ such that:\\ \mbox{$\qquad \qquad \mathbf{S} = \{ \mathbf{s} | \mathbf{s} \subset \{1,2, \ldots p\}, \; |\mathbf{s}| = p-k \}$}.
\STATE  For each $\mathbf{s} \in \mathbf{S}$, run a Kalman filter that uses all sensors indexed by $\mathbf{s}$ and returns estimate $\hat{\mathbf{x}}_{\mathbf{s}}(t) \in \R^n$.
\STATE For each set $\mathbf{s} \in \mathbf{S}$, enumerate all subsets of size $\theta$ and index them by $i$. Let $\mu_i$ be the observability index associated with sensor subset $i$.
For each subset $i$ of $\mathbf{s}$ (subset of size $\theta$), calculate the \textit{block} residue:
$$ \mathbf{r}_i(t) = \begin{bmatrix} \mathbf{y}_i(t) \\ \mathbf{y}_i(t+1) \\ \vdots \\ \mathbf{y}_i(t+\mu_i -1 ) \end{bmatrix} - \mathbf{O}_i \hat{\mathbf{x}}_{\mathbf{s}}(t)\quad  \forall t \in G.$$
\STATE 
Pick the set $\mathbf{s}^{\star} \in \mathbf{S}$ which satisfies the following block residue test for
each subset $i$ of $\mathbf{s}^\star$ (subset of size $\theta$).
Partition $G$ into $\mu_i$ groups $G_0, G_1, \ldots G_{{\mu_i}-1}$ of size $N_B$ such that $G_l = \{ t |  \left( \left( t-t_1 \right) \; mod \; {\mu_i} \right) = l \}$ and check that for each $G_l$:
\begin{align}
& \frac{1}{N_B} \sum_{t \in G_l}  tr \left(   \mathbf{O}^{\dagger}_i  \mathbf{r}_i(t)  \mathbf{r}^T_i(t) \mathbf{O}^{\dagger T}_i  \right)  \nonumber \\
& \leq  P_{opt, \mathbf{s}^\star} + tr \left( \mathbf{O}^{\dagger}_i \mathbf{M}_{\mu_i} \mathbf{O}^{\dagger T}_i \right) + \epsilon, \label{eq:block_residue_test_prediction}
\end{align}
where $\epsilon \geq 0 $ is a design parameter which can be made arbitrarily small for large enough $N_B$.

\STATE Return $\mathbf{s}^{\star}$ and $\hat{\mathbf{x}}(t) := \hat{\mathbf{x}}_{\mathbf{s}^\star}(t) \quad \forall t  \in G$.
\end{algorithmic}
\label{alg:vector_prediction}
\end{algorithm} 
Similar to the scalar setting, Algorithm \ref{alg:vector_prediction} runs a bank of $p\choose{p-k}$ Kalman filters in parallel.
For each distinct set $\mathbf{s}$ of $p-k$ sensors, the corresponding Kalman filter fuses all the measurements from these sensors in order to calculate an estimate $\hat{\mathbf{x}}_{\mathbf{s}}(t)$.
For a sensor set $\mathbf{s}$ of size $p-k$ to satisfy the block residue test, each of its $p-k \choose \theta$ subsets should satisfy \eqref{eq:block_residue_test_prediction} for each group $G_l$.
If a set $\mathbf{s}^\star$ satisfies the residue test, it is declared good and the corresponding Kalman estimate $\hat{\mathbf{x}}_{\mathbf{s}^\star}(t)$ is used as the state estimate for the given time window.
Intuitively, the residue test checks if the outputs from every \textit{observable} sensor subset of size $\theta$ within set $\mathbf{s}$ are \textit{consistent} with the corresponding Kalman estimate over the time window $G$. We analyze the performance of Algorithm~\ref{alg:vector_prediction} in Appendix~\ref{sec:vector_prediction_analysis}.
\section{Sparse observability: Coding theoretic view} \label{sec:sparse_proof}
In this section, we revisit the sparse observability condition \eqref{eq:theta_condition} against a $k$-adversary and give a coding theoretic interpretation for the same.
We first describe our interpretation for a linear system,
and then discuss how it can be generalized for non-linear systems.

\par Consider the linear dynamical system in \eqref{eq:vector_system_model} without the process and sensor noise (\emph{i.e.}, $\mathbf{x}\left(t+1\right) = \mathbf{Ax}(t), \; \mathbf{y}(t) = \mathbf{C} \mathbf{x}(t)+ \pmb{\phi}(t)$).
If the system's initial state is $\mathbf{x}(0) \in \mathbb{R}^n$ and the system is $\theta$-sparse observable, then clearly in the absence of sensor attacks, by observing the outputs from any $\theta$ out of $p$ sensors for $n$ time instants ($t =0,1,\ldots  n-1$) we can exactly recover $\mathbf{x}(0)$ and hence, \textit{exactly} estimate the state of the plant.
A coding theoretic view of this can be given as follows.
Consider the outputs from sensor $d \in \{1,2,\ldots p\}$ for $n$ time instants as a symbol $\pazocal{Y}_d \in \mathbb{R}^n$.
Thus, in the (symbol) observation vector $\pazocal{Y}= \begin{bmatrix}\pazocal{Y}_1 & \pazocal{Y}_2  \ldots \pazocal{Y}_p \end{bmatrix}$,
due to $\theta$-sparse observability,
any $\theta$ symbols are sufficient (in the absence of attacks) to recover the initial state $\mathbf{x}(0)$.
Now, let us consider the case of a $k$-adversary which can arbitrarily corrupt any $k$ sensors.
In the coding theoretic view, this corresponds to arbitrarily corrupting any $k$ (out of $p$) symbols in the observation vector.
Intuitively,
based on the relationship between error correcting codes and the Hamming distance between codewords in classical coding theory \cite{blahut}, one can expect the recovery of the initial state despite such corruptions to depend on the (symbol) Hamming distance
between the observation vectors corresponding to two distinct initial states (say $\mathbf{x}^{(1)}(0)$ and $\mathbf{x}^{(2)}(0)$ with $\mathbf{x}^{(1)}(0) \neq \mathbf{x}^{(2)}(0)$).
In this context,
the following lemma relates $\theta$-sparse observability
to the minimum Hamming distance between observation vectors in the absence of attacks;
this leads to a (tight) bound on
the number of attacked sensors that can be tolerated for state estimation.

\begin{lemma} \label{lemma:distance}
For a $\theta$-sparse observable system with $p$ sensors, the minimum (symbol) Hamming distance between observation vectors corresponding to distinct initial states is $p-\theta+1$. 
\end{lemma}

\begin{proof}
Consider observation vectors $\pazocal{Y}^{(1)}$ and $\pazocal{Y}^{(2)}$ corresponding to distinct initial states $\mathbf{x}^{(1)}(0)$ and $\mathbf{x}^{(2)}(0)$. 
Due to $\theta$-sparse observability,
at most $\theta-1$ symbols in $\pazocal{Y}^{(1)}$ and $\pazocal{Y}^{(2)}$ can be identical;
if any $\theta$ of the symbols are identical, this would imply $\mathbf{x}^{(1)}(0) = \mathbf{x}^{(2)}(0)$.
Hence, the (symbol) Hamming distance between the observation vectors
$\pazocal{Y}^{(1)}$ and $\pazocal{Y}^{(2)}$
(corresponding to $\mathbf{x}^{(1)}(0)$ and $\mathbf{x}^{(2)}(0)$)
is at least $p-(\theta-1) = p - \theta +1$ symbols.
Furthermore, there exists a pair of initial states $\left( \mathbf{x}^{(1)}(0), \mathbf{x}^{(2)}(0)\right)$,
such that the corresponding observation vectors $\pazocal{Y}^{(1)}$ and $\pazocal{Y}^{(2)}$ are identical in exactly $\theta-1$ symbols\footnote{If there is no such pair of initial states,
the initial state can be recovered by observing any $\theta-1$ sensors.
By definition, in a $\theta$-sparse observable system, $\theta$ is the smallest positive integer, such that the initial state can be recovered by observing any $\theta$ sensors.}
and differ in the rest $p-\theta+1$ symbols.
Hence, the minimum (symbol) Hamming distance between the observation vectors is  $p-\theta+1$.
\end{proof}

\par The above lemma connects the problem of
state estimation with sensor attacks in a dynamical system
to error correction in classical coding theory.
Since the minimum Hamming distance between the observation vectors corresponding to
distinct initial states is $p-\theta+1$, we can correct up to $k < \frac{p-\theta+1}{2}$ sensor corruptions;
this is equivalent to the condition $\theta \leq p-2k$,
which is precisely the sparse observability condition required against a $k$-adversary\footnote{In addition, since the minimum Hamming distance is $p-\theta+1$, we can detect attacks up to $(p-\theta+1) - 1 = p-\theta$ sensor corruptions.}.
It should be noted that a $k$-adversary can attack \textit{any} set of $k$ (out of $p$) sensors,
and the condition $k < \frac{p-\theta+1}{2}$ is both necessary and sufficient for exact state estimation despite such attacks.
When $k \geq \frac{p-\theta+1}{2}$,
it is straightforward to show a scenario where the observation vector (after attacks) can be explained by multiple initial states,
and hence exact state estimation is not possible.
The following example illustrates such an attack scenario in view of the coding theoretic interpretation discussed above.

\begin{example} \label{example:coding_interpretation}
Consider a $\theta$-sparse observable system with $\theta=2$, number of sensors $p=5$, and a $k$-adversary with $k=2$.
Clearly, the condition $k < \frac{p-\theta+1}{2}$ is not satisfied in this example.
Let $\mathbf{x}^{(1)}(0)$ and $\mathbf{x}^{(2)}(0)$ be distinct initial states,
such that the corresponding observation vectors $\pazocal{Y}^{(1)}$ and $\pazocal{Y}^{(2)}$ have (minimum) Hamming distance $p-\theta+1 = 4$ symbols.
Figure~\ref{fig:coding_example} depicts the observation vectors $\pazocal{Y}^{(1)}$ and $\pazocal{Y}^{(2)}$, and for the sake of this example,
we assume that the observation vectors have the same first symbol (\emph{i.e.}, $\pazocal{Y}^{(1)}_1 = \pazocal{Y}_1^{(2)} = \pazocal{Y}_1$) and differ in the rest $4$ symbols
(hence, a Hamming distance of $4$).
\begin{figure}[!ht]
\begin{center}
\includegraphics[scale=0.65]{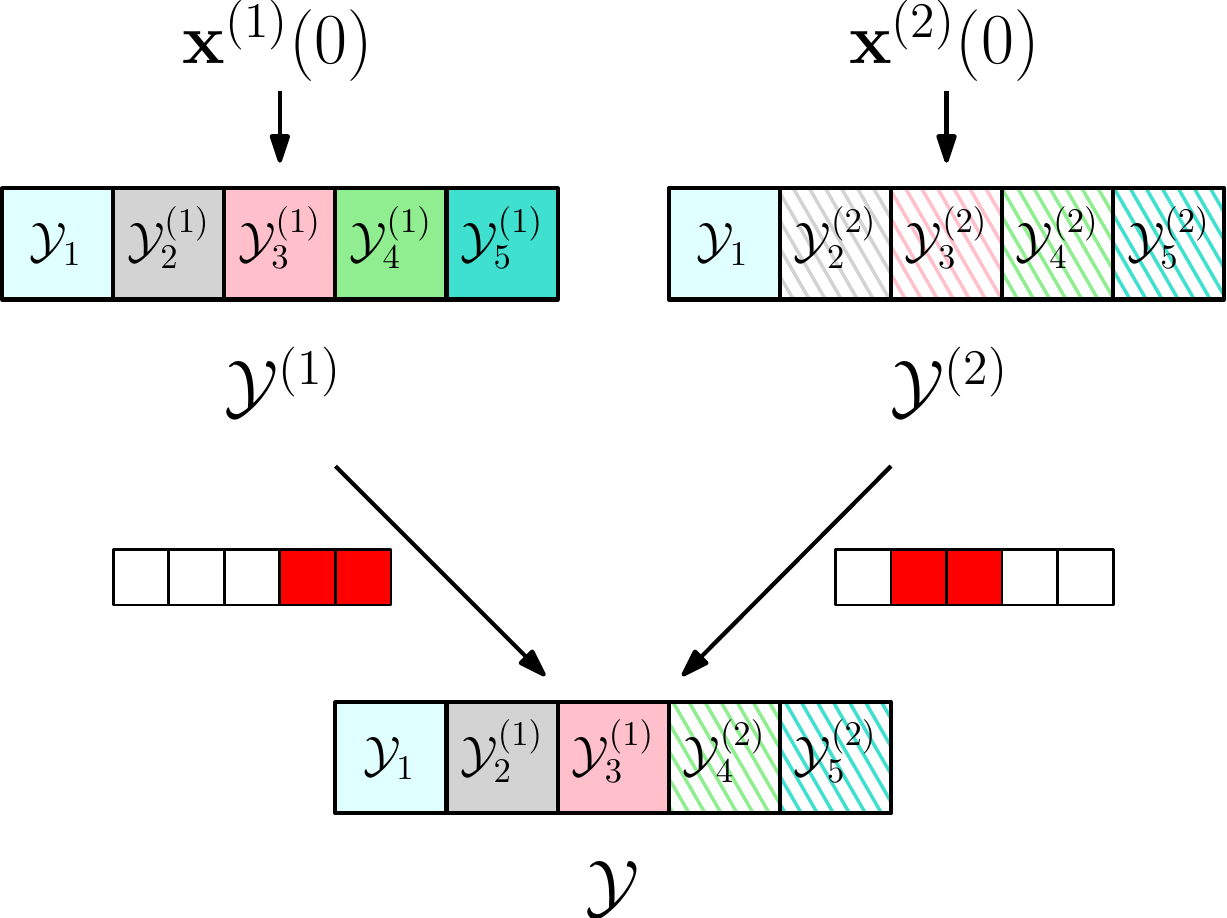}
\caption{Example with $\theta = 2$, $p=5$ and $k=2$.
For distinct initial states $\mathbf{x}^{(1)}(0)$ and $\mathbf{x}^{(2)}(0)$, the corresponding observation vectors are $\pazocal{Y}^{(1)}$ and $\pazocal{Y}^{(2)}$.
Both $\pazocal{Y}^{(1)}$ and $\pazocal{Y}^{(2)}$ have the same first symbol, but differ in the rest four symbols. Given (attacked) observation vector $\pazocal{Y} = \left [ \pazocal{Y}_1 \; \;  \pazocal{Y}_2^{(1)} \; \; \pazocal{Y}_3^{(1)} \; \;  \pazocal{Y}_4^{(2)} \; \; \pazocal{Y}_5^{(2)} \right]$, there are two possibilities for the initial state: (a) $\mathbf{x}^{(1)}(0)$ with attacks on sensors
$4$ and $5$, or (b) $\mathbf{x}^{(2)}(0)$ with attacks on sensors
$2$ and $3$.}
\label{fig:coding_example}
\end{center}
\end{figure}
Now, as shown in Figure~\ref{fig:coding_example}, suppose the observation vector after attacks was
$\pazocal{Y} = \left [ \pazocal{Y}_1 \; \;  \pazocal{Y}_2^{(1)} \; \; \pazocal{Y}_3^{(1)} \; \;  \pazocal{Y}_4^{(2)} \; \; \pazocal{Y}_5^{(2)} \right]$.
Clearly, there are two possible explanations for this (attacked) observation vector: (a) the initial state was $\mathbf{x}^{(1)}(0)$ and sensors $4$ and $5$ were attacked,
or (b) the initial state was $\mathbf{x}^{(2)}(0)$ and sensors $2$ and $3$ were attacked. Since there are two possibilities,
we cannot estimate the initial state exactly given the attacked observation vector.
This example can be easily generalized to show the necessity of the condition $k < \frac{p-\theta+1}{2}$.
\end{example}

\par For (noiseless) non-linear systems, by analogously defining $\theta$-sparse observability, the same coding theoretic interpretation holds.
Hence, this leads to an alternative proof for the necessary and sufficient conditions for secure state estimation in any noiseless dynamical system.

\vspace{-1mm}

\bibliographystyle{IEEEtran}
\bibliography{bibliography2} 

\begin{thebibliography}{10}
\providecommand{\url}[1]{#1}
\csname url@samestyle\endcsname
\providecommand{\newblock}{\relax}
\providecommand{\bibinfo}[2]{#2}
\providecommand{\BIBentrySTDinterwordspacing}{\spaceskip=0pt\relax}
\providecommand{\BIBentryALTinterwordstretchfactor}{4}
\providecommand{\BIBentryALTinterwordspacing}{\spaceskip=\fontdimen2\font plus
\BIBentryALTinterwordstretchfactor\fontdimen3\font minus
  \fontdimen4\font\relax}
\providecommand{\BIBforeignlanguage}[2]{{%
\expandafter\ifx\csname l@#1\endcsname\relax
\typeout{** WARNING: IEEEtran.bst: No hyphenation pattern has been}%
\typeout{** loaded for the language `#1'. Using the pattern for}%
\typeout{** the default language instead.}%
\else
\language=\csname l@#1\endcsname
\fi
#2}}
\providecommand{\BIBdecl}{\relax}
\BIBdecl

\bibitem{Bullo_CSM}
\BIBentryALTinterwordspacing
F.~Pasqualetti, F.~Dorfler, and F.~Bullo, ``Control-theoretic methods for
  cyber-physical security,'' \emph{IEEE Control Systems Magazine}, Aug. 2014,
  to appear. [Online]. Available:
  \url{http://motion.me.ucsb.edu/pdf/2013u-pdb.pdf}
\BIBentrySTDinterwordspacing

\bibitem{Hamza_TAC}
H.~Fawzi, P.~Tabuada, and S.~Diggavi, ``Secure estimation and control for
  cyber-physical systems under adversarial attacks,'' \emph{IEEE Transactions
  on Automatic Control}, vol.~59, no.~6, pp. 1454--1467, June 2014.

\bibitem{YasserETPGarXiv}
\BIBentryALTinterwordspacing
Y.~{Shoukry} and P.~{Tabuada}, ``{Event-triggered state observers for sparse
  sensor noise/attacks},'' \emph{arXiv pre-print}, Sep. 2013. [Online].
  Available: \url{http://arxiv.org/abs/1309.3511}
\BIBentrySTDinterwordspacing

\bibitem{shaunak_CDC14}
S.~Mishra, N.~Karamchandani, P.~Tabuada, and S.~Diggavi, ``Secure state
  estimation and control using multiple (insecure) observers,'' in \emph{IEEE
  Conference on Decision and Control (CDC)}, 2014.

\bibitem{Yasser_SMT}
Y.~{Shoukry}, P.~{Nuzzo}, A.~{Puggelli}, A.~L. {Sangiovanni-Vincentelli}, S.~A.
  {Seshia}, and P.~{Tabuada}, ``{Secure state estimation for cyber physical
  systems under sensor attacks: a satisfiability modulo theory approach},''
  \emph{arXiv pre-print}, Dec. 2014.

\bibitem{Joao_ACC}
M.~S. {Chong}, M.~{Wakaiki}, and J.~P. {Hespanha}, ``Observability of linear
  systems under adversarial attacks,'' in \emph{American Control Conference
  (ACC)}, 2015.

\bibitem{Pajic_ICCPS}
M.~Pajic, J.~Weimer, N.~Bezzo, P.~Tabuada, O.~Sokolsky, I.~Lee, and G.~Pappas,
  ``Robustness of attack-resilient state estimators,'' in \emph{ACM/IEEE
  International Conference on Cyber-Physical Systems (ICCPS)}, 2014.

\bibitem{YilinAllerton}
Y.~Mo and B.~Sinopoli, ``Secure control against replay attacks,'' in
  \emph{Allerton Conference on Communication, Control, and Computing}, 2009.

\bibitem{Bai_Gupta}
C.-Z. Bai and V.~Gupta, ``On kalman filtering in the presence of a compromised
  sensor: fundamental performance bounds,'' in \emph{American Control
  Conference (ACC)}, 2014.

\bibitem{BoydKF}
J.~Mattingley and S.~Boyd, ``Real-time convex optimization in signal
  processing,'' \emph{IEEE Signal Processing Magazine}, vol.~27, no.~3, pp.
  50--61, May 2010.

\bibitem{Georgios_TSP2}
S.~Farahmand, G.~B. Giannakis, and D.~Angelosante, ``Doubly robust smoothing of
  dynamical processes via outlier sparsity constraints,'' \emph{IEEE Trans. on
  Signal Processing}, vol.~59, no.~10, pp. 4529--4543, Oct. 2011.

\bibitem{kailath_book}
T.~Kailath, A.~Sayed, and B.~Hassibi, \emph{Linear Estimation}.\hskip 1em plus
  0.5em minus 0.4em\relax Prentice Hall, 2000.

\bibitem{blahut}
R.~Blahut, \emph{Algebraic Codes for Data Transmission}.\hskip 1em plus 0.5em
  minus 0.4em\relax Cambridge University Press, 2003.

\bibitem{trace_ineq_wang}
S.-D. Wang, T.-S. Kuo, and C.-F. Hsu, ``Trace bounds on the solution of the
  algebraic matrix {Riccati} and {Lyapunov} equation,'' \emph{IEEE Transactions
  on Automatic Control}, vol.~31, no.~7, pp. 654--656, Jul 1986.

\end{thebibliography}


\appendix

\subsection{Algorithm~\ref{alg:vector_prediction}: performance analysis} \label{sec:vector_prediction_analysis}
In this section,
we analyze the performance of Algorithm~\ref{alg:vector_prediction}.
Similar to the analysis done for the scalar setting in Section~\ref{sec:scalar_state},
we first derive a bound using LLN, and then analyze the \textit{cross term} in the bound to obtain final guarantees on the state estimation error in the presence of attacks.
The details of the analysis are described below.

\par Consider the set $\mathbf{s}$ of $p-k$ sensors which are not attacked by the $k$-adversary.
For such a set $\mathbf{s}$,
the block residue $\mathbf{r}_i(t)$ for a subset $i$ of $\mathbf{s}$ (subset of size $\theta$) can be expressed as shown below:
\begin{align}
 \mathbf{r}_i(t)  
 &= \begin{bmatrix} \mathbf{y}_i(t) \\ \vdots \\ \mathbf{y}_i(t+{\mu_i} -1 ) \end{bmatrix} - \mathbf{O}_i \hat{\mathbf{x}}_{\mathbf{s}}(t) \nonumber \\
& = \mathbf{O}_i \mathbf{x}(t) 
+ \mathbf{J}_i\begin{bmatrix} \mathbf{w}(t) \\ \mathbf{w}(t+1) \\ \vdots \\ \mathbf{w}(t+ {\mu_i} -2) \end{bmatrix}
+ \begin{bmatrix} \mathbf{v}_i(t) \\ \mathbf{v}_i (t+1) \\ \vdots \\ \mathbf{v}_i(t + {\mu_i} -1)\end{bmatrix}   \nonumber  \\ & \quad - \mathbf{O}_i \hat{\mathbf{x}}_{\mathbf{s}}(t)
\nonumber \\
& =\mathbf{O}_i \left (  \mathbf{x}(t) -\hat{\mathbf{x}}_{\mathbf{s}}(t)  \right) + \left ( \mathbf{J}_i \mathbf{w}_{t:t+{\mu_i}-2} + \mathbf{v}_{i,t:t+{\mu_i}-1} \right ) \label{eq:vi_def} \\
& =  \mathbf{O}_i  \left(  \mathbf{x}(t) -\hat{\mathbf{x}}_{\mathbf{s}}(t)  \right) 
+ \mathbf{z}_{i,t:t+{\mu_i}-1},
\end{align}
and assuming that the Kalman filter corresponding to sensor set $\mathbf{s}$ is in steady state: 
\begin{align}
\mathbb{E} \left( tr\left( \mathbf{O}^{\dagger}_i \mathbf{r}_i(t) \mathbf{r}^T_i(t) \mathbf{O}^{\dagger T}_i\right) \right)
& \stackrel{(a)}= P_{opt,\mathbf{s}}  + tr \left( \mathbf{O}^{\dagger }_i \mathbf{M}_{\mu_i} \mathbf{O}^{\dagger T}_i \right),  \nonumber
\end{align}where (a) follows from $\mathbf{M}_{\mu_i} = \mathbb{E} \left ( \mathbf{z}_{i,t:t+{\mu_i}-1} \mathbf{z}^T_{i,t:t+{\mu_i}-1} \right)  = \sigma^2_w \mathbf{J}_i \mathbf{J}^T_i + \sigma^2_v \mathbf{I}_{\mu_i }$.
Hence, due to LLN, the block residue test \eqref{eq:block_residue_test_prediction} will be satisfied w.h.p. for at least this set of good sensors and w.h.p. the algorithm will not return an empty set.
Also, the estimate $\hat{\mathbf{x}}_{\mathbf{s}}(t)$ from this set of good sensors trivially satisfies the error bound \eqref{eq:sse}.
But, since the algorithm can return any set of size $p-k$ which satisfies the block residue test,
it may be possible that some of the sensors in the returned set are corrupt.
In the remainder of our analysis, we show that for \textit{any} set returned by the algorithm, the corresponding Kalman estimate achieves the error bound \eqref{eq:sse}.

\par Suppose the algorithm returns a set $\mathbf{s}$ of $p-k$ sensors.
Since $ \theta \leq p -2k$ (sparse observability condition), there exists a subset of $\theta$ good sensors in $\mathbf{s}$. The following can be inferred when the
block residue test \eqref{eq:block_residue_test_prediction} is satisfied for such a subset $i$ (of size $\theta$):
{\allowdisplaybreaks[4]
\begin{align}
& tr \left ( \frac{1}{N_B} \sum_{t \in G_l} \mathbf{O}^{\dagger}_i \mathbf{r}_i(t) \mathbf{r}^T_i(t) \mathbf{O}^{\dagger T}_i \right) \nonumber \\
&\stackrel{(a)}= \frac{1}{N_B} \sum_{t \in G_l} \mathbf{e}^T(t)\mathbf{e}(t)
+ \frac{1}{N_B} \sum_{t \in G_l} tr\left ( \mathbf{O}^{\dagger }_i\mathbf{z}_{i,t:t+{\mu_i}-1} \mathbf{z}^T_{i,t:t+{\mu_i}-1} \mathbf{O}^{\dagger T}_i\right ) \nonumber \\ 
& \quad 
+ \frac{2}{N_B} \sum_{t \in G_l} \mathbf{e}^T(t) \mathbf{O}^{\dagger }_i\mathbf{z}_{i,t:t+{\mu_i}-1}
\nonumber \\
& \stackrel{(b)}\leq P_{opt,\mathbf{s}}  + tr\left ( \mathbf{O}^{\dagger}_i \mathbf{M}_{\mu_i} \mathbf{O}^{\dagger T}_i \right) + \epsilon,
\end{align}}where $\mathbf{e}(t)$ in (a) is the state estimation error at time $t$ (in the presence of a $k$-adversary) when $\hat{\mathbf{x}}_{\mathbf{s}}(t)$ is used as the state estimate, and (b) follows from the block residue test \eqref{eq:block_residue_test_prediction}.
Using (a) and (b) above, for any $\epsilon > 0$ there exists a large enough $N_B$ such that:
\begin{align}
& \frac{1}{N_B} \sum_{t \in G_l} \mathbf{e}^T(t)\mathbf{e}(t)
+ \frac{2}{N_B} \sum_{t \in G_l} \mathbf{e}^T(t) \mathbf{O}^{\dagger }_i\mathbf{z}_{i,t:t+{\mu_i}-1} \label{eq:cross_term_vector_prediction} \\
& \leq    \left | tr\left ( \mathbf{O}^{\dagger }_i \mathbf{M}_{\mu_i} \mathbf{O}^{\dagger T}_i \right) - \frac{1}{N_B} \sum_{t \in G_l} tr\left ( \mathbf{O}^{\dagger }_i\mathbf{z}_{i,t:t+{\mu_i}-1} \mathbf{z}^T_{i,t:t+{\mu_i}-1} \mathbf{O}^{\dagger T}_i\right ) \right| \nonumber\\ & \quad  + 
P_{opt,\mathbf{s}} 
+ \epsilon \nonumber \\
& \stackrel{(c)}\leq  P_{opt, \mathbf{s}}  + 2 \epsilon, \label{eq:vector_state_prediction_bound1}
\end{align}
where (c) follows w.h.p. from LLN; for different time indices in $G_l$, $tr\left ( \mathbf{O}^{\dagger }_i\mathbf{z}_{i,t:t+{\mu_i}-1} \mathbf{z}^T_{i,t:t+{\mu_i}-1} \mathbf{O}^{\dagger T}_i\right )$ corresponds to i.i.d. realizations of the same random variable. Along the lines of the analysis done in the scalar setting in Section~\ref{sec:scalar_state},
we can show that the cross term $ \frac{2}{N_B} \sum_{t \in G_l} \mathbf{e}^T(t) \mathbf{O}^{\dagger }_i\mathbf{z}_{i,t:t+{\mu_i}-1}$ in \eqref{eq:cross_term_vector_prediction} has zero mean and vanishingly small variance as $N_B \rightarrow \infty$; this leads to the required bound on $\frac{1}{N} \sum_{t \in G} \mathbf{e}^T(t)\mathbf{e}(t)$. To complete our analysis we calculate the mean and the variance of the cross term $ \frac{2}{N_B} \sum_{t \in G_l} \mathbf{e}^T(t) \mathbf{O}^{\dagger }_i\mathbf{z}_{i,t:t+{\mu_i}-1}$ as shown below.

\par The mean of $ \frac{2}{N_B} \sum_{t \in G_l} \mathbf{e}^T(t) \mathbf{O}^{\dagger }_i\mathbf{z}_{i,t:t+{\mu_i}-1}$ can be computed as shown below:
\begin{align}
& \mathbb{E} \left ( \frac{2}{N_B} \sum_{t \in G_l} \mathbf{e}^T(t)  \mathbf{O}^{\dagger}_i\mathbf{z}_{i,t:t+{\mu_i}-1}   \right)  \nonumber \\
& 
\stackrel{(a)}=  \frac{2}{N_B} \sum_{t \in G_l} \mathbb{E} \left ( \mathbf{e}^T(t)  \right)  \mathbb{E} \left (  \mathbf{O}^{\dagger}_i\mathbf{z}_{i,t:t+{\mu_i}-1} \right ) = 0, \label{eq:zero_mean_vector_prediction}
\end{align}where (a) follows from the independence of $\mathbf{e}(t)$ from $\mathbf{z}_{i,t:t+{\mu_i}-1}$.
This is true since both $\mathbf{x}(t)$ and $\hat{\mathbf{x}}_{\mathbf{s}}(t)$ are
independent\footnote{The adversary's corruptions till time $t-1$ can influence $\hat{\mathbf{x}}_{\mathbf{s}}(t)$ which is based on outputs till time $t-1$. Due to assumption (A1), the adversary's corruptions till time $t-1$ are independent of $\mathbf{w}(t)$ and hence  $\hat {\mathbf{x}}_{\mathbf{s}}(t)$ is independent of $\mathbf{w}(t)$. Also, $\mathbf{x}(t)$ is independent of $\mathbf{w}(t)$. Due to assumption (A2), $\hat{\mathbf{x}}_{\mathbf{s}}(t)$ is independent of $\mathbf{v}_i(t)$.}
of $\mathbf{w}(t)$ and $\mathbf{v}_i(t)$.
Also, using \eqref{eq:zero_mean_vector_prediction}
and taking the expectation in \eqref{eq:cross_term_vector_prediction}:
\begin{align}
 \mathbb{E} \left ( \frac{1}{N_B} \sum_{t \in G_l} \mathbf{e}^T(t)\mathbf{e}(t) \right ) &  \leq P_{opt, \mathbf{s}}  + 2\epsilon \label{eq:expectation_bound_general_case_prediction}.
\end{align}
\par Now, we will show that the variance of the cross term $\frac{2}{N_B} \sum_{t \in G_l} \mathbf{e}^T(t) \mathbf{O}^{\dagger}_i\mathbf{z}_{i,t:t+{\mu_i}-1} $ is vanishingly small as $N_B \rightarrow \infty$. For any $\epsilon_1 > 0$, there exists a large enough $N_B$ such that:
{\allowdisplaybreaks[4]
\begin{align}
& Var \left( \frac{1}{N_B} \sum_{t \in G_l} \mathbf{e}^T(t) \mathbf{O}^{\dagger}_i\mathbf{z}_{i,t:t+{\mu_i}-1} \right) \nonumber \\
& = \mathbb{E} \left ( \left( \frac{1}{N_B} \sum_{t \in G_l} \mathbf{e}^T(t) \mathbf{O}^{\dagger}_i\mathbf{z}_{i,t:t+{\mu_i}-1} \right)^2 \right)  \nonumber \\ & \quad - \left (\mathbb{E}\left( \frac{1}{N_B} \sum_{t \in G_l} \mathbf{e}^T(t) \mathbf{O}^{\dagger}_i\mathbf{z}_{i,t:t+{\mu_i}-1} \right) \right)^2 \nonumber \\
& \stackrel{(a)}= \mathbb{E} \left ( 
\left( \frac{1}{N_B} \sum_{t \in G_l} \mathbf{e}^T(t) \mathbf{O}^{\dagger}_i\mathbf{z}_{i,t:t+{\mu_i}-1} \right)^2 \right)  \nonumber \\
& = \mathbb{E}\left( \frac{1}{N^2_B} \sum_{t \in G_l} \mathbf{e}^T(t) \mathbf{O}^{\dagger}_i\mathbf{z}_{i,t:t+{\mu_i}-1} \mathbf{e}^T(t) \mathbf{O}^{\dagger}_i\mathbf{z}_{i,t:t+{\mu_i}-1} \right) \nonumber \\
& \quad + \mathbb{E}\left( \frac{2}{N^2_B} \sum_{t,t' \in G_l, \; t < t'} \mathbf{e}^T(t) \mathbf{O}^{\dagger}_i\mathbf{z}_{i,t:t+{\mu_i}-1} \mathbf{e}^T(t') \mathbf{O}^{\dagger}_i\mathbf{z}_{i,t':t'+{\mu_i}-1}   \right) \nonumber\\
& \stackrel{(b)} = \mathbb{E}\left( \frac{1}{N^2_B} \sum_{t \in G_l} \mathbf{e}^T(t) \mathbf{O}^{\dagger}_i\mathbf{z}_{i,t:t+{\mu_i}-1} \mathbf{e}^T(t) \mathbf{O}^{\dagger}_i\mathbf{z}_{i,t:t+{\mu_i}-1} \right) +  \nonumber \\
& \quad  \frac{2}{N^2_B} \sum_{t,t' \in G_l, \; t < t'} \mathbb{E}\left( \mathbf{e}^T(t) \mathbf{O}^{\dagger}_i \mathbf{z}_{i,t:t+{\mu_i}-1} \mathbf{e}^T(t') \mathbf{O}^{\dagger}_i \right)
\mathbb{E}\left( \mathbf{z}_{i,t':t'+{\mu_i}-1} \right) \nonumber\\
& \stackrel{(c)}= \mathbb{E}\left( \frac{1}{N^2_B} \sum_{t \in G_l} \mathbf{e}^T(t) \mathbf{O}^{\dagger}_i\mathbf{z}_{i,t:t+{\mu_i}-1} \mathbf{e}^T(t) \mathbf{O}^{\dagger}_i\mathbf{z}_{i,t:t+{\mu_i}-1} \right) \nonumber\\
& = \mathbb{E}\left( \frac{1}{N^2_B} \sum_{t \in G_l} \mathbf{e}^T(t) \mathbf{O}^{\dagger}_i\mathbf{z}_{i,t:t+{\mu_i}-1} \left( \mathbf{O}^{\dagger}_i\mathbf{z}_{i,t:t+{\mu_i}-1} \right)^T \mathbf{e}(t) \right) \nonumber\\
& \stackrel{(d)}= \mathbb{E} \left( \frac{1}{N^2_B} \sum_{t \in G_l} tr \left( \mathbf{e}^T(t) \mathbf{O}^{\dagger}_i\mathbf{z}_{i,t:t+{\mu_i}-1} \left( \mathbf{O}^{\dagger}_i\mathbf{z}_{i,t:t+{\mu_i}-1} \right)^T \mathbf{e}(t) \right) \right) \nonumber\\
& = \mathbb{E} \left( \frac{1}{N^2_B} \sum_{t \in G_l} tr \left( \mathbf{O}^{\dagger}_i\mathbf{z}_{i,t:t+{\mu_i}-1} \left( \mathbf{O}^{\dagger}_i\mathbf{z}_{i,t:t+{\mu_i}-1} \right)^T \mathbf{e}(t) \mathbf{e}^T(t) \right) \right) \nonumber\\
& \stackrel{(e)}= \frac{ 1}{N^2_B}\sum_{t \in G_l} tr \left( \mathbb{E} \left( \mathbf{O}^{\dagger}_i\mathbf{z}_{i,t:t+{\mu_i}-1} \mathbf{z}^T_{i,t:t+{\mu_i}-1} \mathbf{O}^{\dagger T }_i  \right)
\mathbb{E}
\left(\mathbf{e}(t) \mathbf{e}^T(t) \right) \right) \nonumber\\
& \stackrel{(f)}\leq \frac{1}{N^2_B} \sum_{t \in G_l} \lambda^* tr \left( \mathbb{E}
\left(\mathbf{e}(t) \mathbf{e}^T(t) \right) \right) \nonumber\\
& = \frac{ \lambda^*}{N_B} \mathbb{E}
\left( \frac{1}{N_B} \sum_{t \in G_l} tr \left( \mathbf{e}(t) \mathbf{e}^T(t) \right) \right) \nonumber\\
& = \frac{ \lambda^*}{N_B} \mathbb{E}
\left( \frac{1}{N_B} \sum_{t \in G_l} \mathbf{e}^T(t) \mathbf{e}(t) \right) \nonumber\\
& \stackrel{(g)} \leq \epsilon_1,
\end{align}}where (a) follows from \eqref{eq:zero_mean_vector_prediction}, (b) follows from the independence of $ \mathbf{z}_{i,t':t'+{\mu_i}-1} $ from $ \mathbf{e}^T(t) \mathbf{O}^{\dagger}_i \mathbf{z}_{i,t:t+{\mu_i}-1} \mathbf{e}^T(t') \mathbf{O}^{\dagger}_i$ for $t'>t$,
(c) follows from $\mathbb{E}\left( \mathbf{z}_{i,t':t'+{\mu_i}-1} \right) =\mathbf{0}$,
(d) follows from $\mathbf{e}^T(t) \mathbf{O}^{\dagger}_i\mathbf{z}_{i,t:t+{\mu_i}-1}$ being a scalar,
(e) follows from the independence of $ \mathbf{z}_{i,t:t+{\mu_i}-1} $ from $\mathbf{e}(t)$,
(f) follows from Lemma~\ref{lemma:eigen_bound} (discussed in Appendix~\ref{sec:trace_ineq}) with eigen value $\lambda^* = \lambda_{max} \left( \mathbb{E} \left( \mathbf{O}^{\dagger}_i\mathbf{z}_{i,t:t+{\mu_i}-1} \left( \mathbf{O}^{\dagger}_i\mathbf{z}_{i,t:t+{\mu_i}-1} \right)^T \right) \right) = \lambda_{max} \left( \mathbf{O}^{\dagger}_i \mathbf{M}_{\mu_i} \mathbf{O}^{\dagger T}_i \right)$
(\emph{i.e.}, $\lambda^*$ is the maximum eigen value of
$\mathbf{O}^{\dagger}_i \mathbf{M}_{\mu_i} \mathbf{O}^{\dagger T}_i$).
Finally, (g) follows from \eqref{eq:expectation_bound_general_case_prediction}.
This completes the variance analysis and clearly the cross term $\frac{2}{N_B} \sum_{t \in G_l} \mathbf{e}^T(t) \mathbf{O}^{\dagger}_i\mathbf{z}_{i,t:t+{\mu_i}-1} $ has vanishingly small variance as $N_B \rightarrow \infty$. As a result, using Chebyshev's inequality and \eqref{eq:vector_state_prediction_bound1}, we have the following bound:
for any $\epsilon_2 > 0$ and $\delta > 0$, there exists a large enough $N_B$ such that:
\begin{align}
\mathbb{P} \left ( \frac{1}{N_B} \sum_{t \in G_l } \mathbf{e}^T(t) \mathbf{e}(t)  \leq P_{opt, \mathbf{s}} +  \epsilon_2 \right) \geq 1-\delta . \label{eq:N_B_final_bound}
\end{align}
Since $\frac{1}{N_B} \sum_{t \in G_l } \mathbf{e}^T(t) \mathbf{e}(t) \leq P_{opt, \mathbf{s}} +  \epsilon_2 \quad \forall l \in\{0,1,\ldots \mu_i-1\}$ implies $\frac{1}{N} \sum_{t \in G } \mathbf{e}^T(t) \mathbf{e}(t)  \leq P_{opt, \mathbf{s}} + \epsilon_2$, we have the required bound on $\frac{1}{N} \sum_{t \in G } \mathbf{e}^T(t) \mathbf{e}(t) $ from \eqref{eq:N_B_final_bound} as follows. 
For any $\epsilon_2 > 0$ and $\delta > 0$, there exists a large enough $N$ such that:
\begin{align}
\mathbb{P} \left ( \frac{1}{N} \sum_{t \in G } \mathbf{e}^T(t) \mathbf{e}(t)  \leq  P_{opt, \mathbf{s}} +  \epsilon_2 \right) \geq 1-\delta . \label{eq:N_final_bound}
\end{align}
This completes our performance analysis.

\subsection{Bounds on the trace of product of symmetric matrices} \label{sec:trace_ineq}
A useful lemma from \cite{trace_ineq_wang} providing bounds on the trace of product of symmetric matrices is as follows.
\begin{lemma} \label{lemma:eigen_bound}
If $\mathbf{A}$ and $\mathbf{B}$ are two symmetric matrices in $\mathbb{R}^{n \times n }$, and $\mathbf{B}$ is positive semi-definite (\emph{i.e.}, $\mathbf{B} \succeq 0 $), then the following inequality holds:
\begin{align}
\lambda_{min} \left ( \mathbf{A} \right) tr \left ( \mathbf{B} \right) \leq tr \left ( \mathbf{A} \mathbf{B} \right) \leq \lambda_{max} \left ( \mathbf{A} \right) tr \left ( \mathbf{B} \right).
\end{align}
where $\lambda_{min}\left ( \mathbf{A} \right)$  and $\lambda_{max}\left ( \mathbf{A} \right)$ denote the minimum and maximum eigen values of matrix $\mathbf{A}$.
\end{lemma}

\subsection{Secure state filtering} \label{sec:filtering}
In this section,
for the general linear dynamical system defined in \eqref{eq:vector_system_model},
we study the filtering problem where the goal is to estimate the state at time $t$ based on outputs till time $t$ (in contrast to using outputs till time $t-1$ in the prediction problem).
In the absence of sensor attacks,
using a Kalman filter for state filtering in \eqref{eq:vector_system_model} leads to the optimal (MMSE) error covariance asymptotically \cite{kailath_book}.
The Kalman filter update rule (in steady state) for the filtering problem (without sensor attacks) is as shown below:
\begin{align}
\hat{\mathbf{x}}(t) & = \hat{\mathbf{x}}^{(P)}(t) + \mathbf{L} \left ( \mathbf{y}(t) - \mathbf{C} \hat{\mathbf{x}}^{(P)}(t) \right), \quad 
\hat{\mathbf{x}}^{(P)}(t+1) = \mathbf{A} \hat{\mathbf{x}}(t) ,\label{eq:filtering_update}
\end{align}
where $\hat{\mathbf{x}}(t)$ is the state (filtering) estimate (see \cite{kailath_book} for further details).
The filtering error is defined as $\mathbf{e}(t) = \mathbf{x}(t) - \mathbf{\hat{x}}(t)$, and as shown in \eqref{eq:filtering_update}, the state estimate $\hat{\mathbf{x}}(t)$ at time $t$ depends on the outputs at time $t$. Also, in the absence of sensor attacks, $F_{opt,\mathbf{s}}$ is the trace of steady state (filtering) error covariance matrix obtained by using the Kalman filter on a sensor subset $\mathbf{s} \subseteq \{1,2,\ldots p\}$.

\par For the secure state filtering problem,
we assume that sparse observability condition \eqref{eq:theta_condition}, and assumptions (A3) and (A4) hold against a $k$-adversary. In addition to the notation developed in Section~\ref{sec:vector_state_prediction} for the prediction problem, we will require the following definition: $\mathbf{L}_i \in \mathbb{R}^{ n \times \theta}$ denotes the matrix formed by columns of $\mathbf{L}$ corresponding to sensor subset $i$ of set $\mathbf{s}$ (subset $i$ is of size $\theta$).
The algorithm for secure state filtering (and its analysis) is similar to that for the prediction setting. In the remainder of this section,
we first describe the secure state filtering algorithm
and then analyze its performance.
\paragraph*{Secure state filtering algorithm}
\begin{algorithm}[t]
\caption{\textsc{Secure State Filtering - vector case}}
\begin{algorithmic}[1]
\STATE Enumerate all sets $\mathbf{s} \in \mathbf{S} $ such that:\\ \mbox{$\qquad \qquad \mathbf{S} = \{ \mathbf{s} | \mathbf{s} \subset \{1,2, \ldots p\}, \; |\mathbf{s}| = p-k \}$}.
\STATE  For each $\mathbf{s} \in \mathbf{S}$, run a Kalman filter that uses all sensors indexed by $\mathbf{s}$. The corresponding Kalman (filtering) estimate is denoted by $\hat{\mathbf{x}}_{\mathbf{s}}(t) \in \R^n$.
\STATE For each set $\mathbf{s} \in \mathbf{S}$, enumerate all subsets of size $\theta$ and index them by $i$. Let $\mu_i$ be the observability index associated with sensor subset $i$.
For each subset $i$ of $\mathbf{s}$ (subset of size $\theta$), calculate the \textit{block} residue:
$$ \mathbf{r}_i(t) = \begin{bmatrix} \mathbf{y}_i(t) \\ \mathbf{y}_i(t+1) \\ \vdots \\ \mathbf{y}_i(t+\mu_i -1 ) \end{bmatrix} - \mathbf{O}_i \hat{\mathbf{x}}_{\mathbf{s}}(t)\quad  \forall t \in G.$$

\STATE 
Pick the set $\mathbf{s}^{\star} \in \mathbf{S}$ which satisfies the following block residue test for
each subset $i$ of $\mathbf{s}^\star$ (subset of size $\theta$).
Partition $G$ into $\mu_i$ groups $G_0, G_1, \ldots G_{{\mu_i}-1}$ of size $N_B$ such that $G_l = \{ t | \left ( \left( t-t_1 \right) \; mod \; {\mu_i}  \right) = l \}$ and check that for each $G_l$:
\begin{align}
& \frac{1}{N_B} \sum_{t \in G_l} tr\left( \mathbf{O}^{\dagger}_i \mathbf{r}_i(t) \mathbf{r}^T_i(t) \mathbf{O}^{\dagger T}_i\right)  \nonumber \\
& \leq F_{opt, \mathbf{s}} + tr \left( \mathbf{O}^{\dagger}_i \mathbf{M}_{\mu_i} \mathbf{O}^{\dagger T}_i \right)  \nonumber \\ & \quad-  2 \mathbb{E}\left (  \mathbf{v}^T_i(t) \mathbf{L}^T_i    \mathbf{O}^{\dagger}_i  \mathbf{v}_{i,t:t+{\mu_i}-1}  \right) + \epsilon, \label{eq:block_residue_test_filtering}
\end{align}
where $\epsilon \geq 0 $ is a design parameter which can be made arbitrarily small for large enough $N_B$.

\STATE Return $\mathbf{s}^{\star}$ and $\hat{\mathbf{x}}(t) := \hat{\mathbf{x}}_{\mathbf{s}^\star}(t) \quad \forall t  \in G$.
\end{algorithmic}
\label{alg:vector_filtering}
\end{algorithm}
Algorithm~\ref{alg:vector_filtering} shows the secure state filtering algorithm against a $k$-adversary.
It is same as Algorithm~\ref{alg:vector_prediction} except 
for the usage of Kalman (filtering) estimate and the bound used for the block residue test \eqref{eq:block_residue_test_filtering}; $F_{opt, \mathbf{s}}$ is used instead of $P_{opt, \mathbf{s}}$ and there is an extra term $-  2 \mathbb{E}\left (  \mathbf{v}^T_i(t) \mathbf{L}^T_i    \mathbf{O}^{\dagger}_i  \mathbf{v}_{i,t:t+{\mu_i}-1}  \right)$ (where $\mathbf{v}_{i,t:t+{\mu_i}-1}$ is as defined in \eqref{eq:vi_def}).

\paragraph*{Performance analysis}
The performance analysis is similar to the analysis done for the prediction problem in Appendix~\ref{sec:vector_prediction_analysis} and we describe the details below.

\par Consider the set $\mathbf{s}$ of $p-k$ sensors which are not attacked by the $k$-adversary.
For such a set $\mathbf{s}$,
the block residue $\mathbf{r}_i(t)$ for a subset $i$ of $\mathbf{s}$ (subset of size $\theta$) can be expressed as shown below:
\begin{align}
\mathbf{r}_i(t) & =\mathbf{O}_i \left(  \mathbf{x}(t)  - \hat{\mathbf{x}}_{\mathbf{s}}(t) \right) + \mathbf{z}_{i,t:t+{\mu_i}-1},
\end{align}
and assuming that the Kalman filter corresponding to sensor set $\mathbf{s}$ is in steady state, it can be shown that:
\begin{align}
& \mathbb{E} \left( tr\left( \mathbf{O}^{\dagger}_i \mathbf{r}_i(t) \mathbf{r}^T_i(t) \mathbf{O}^{\dagger T}_i\right) \right) \nonumber \\
& \stackrel{(a)} = F_{opt,\mathbf{s}} + tr \left( \mathbf{O}^{\dagger}_i \mathbf{M}_{\mu_i} \mathbf{O}^{\dagger T}_i \right) 
-  2 \mathbb{E}\left (  \mathbf{v}^T_i(t) \mathbf{L}^T_i    \mathbf{O}^{\dagger}_i  \mathbf{v}_{i,t:t+{\mu_i}-1}  \right)  \label{eq:expected_value_residue_filtering}.
\end{align}
Hence, due to LLN, the block residue test \eqref{eq:block_residue_test_filtering} will be satisfied w.h.p. for at least this set of good sensors and w.h.p. the algorithm will not return an empty set.
Also, the estimate $\hat{\mathbf{x}}_{\mathbf{s}}(t)$ from this set of good sensors trivially satisfies the error bound \eqref{eq:ssf}.
But, since the algorithm can return any set of size $p-k$ which satisfies the block residue test,
it may be possible that some of the sensors in the returned set are corrupt.
In the remainder of our analysis, we show that for \textit{any} set returned by the algorithm, the corresponding Kalman estimate achieves the error bound \eqref{eq:ssf}.

\par Suppose the algorithm returns a set $\mathbf{s}$ of $p-k$ sensors.
Since $ \theta \leq p -2k$ (sparse observability condition), there exists a subset of $\theta$ good sensors in $\mathbf{s}$. The following can be inferred when the block residue test is satisfied for such a subset $i$ (of size $\theta$):
{\allowdisplaybreaks[4]
\begin{align}
& tr \left ( \frac{1}{N_B} \sum_{t \in G_l} \mathbf{O}^{\dagger}_i  \mathbf{r}_i(t) \mathbf{r}^T_i(t) \mathbf{O}^{\dagger T}_i \right) \nonumber \\
&\stackrel{(a)}= \frac{1}{N_B} \sum_{t \in G_l} \mathbf{e}^T(t)\mathbf{e}(t)
+ \frac{1}{N_B} \sum_{t \in G_l} tr\left ( \mathbf{O}^{\dagger }_i\mathbf{z}_{i,t:t+{\mu_i}-1} \mathbf{z}^T_{i,t:t+{\mu_i}-1} \mathbf{O}^{\dagger T}_i\right ) \nonumber \\ & \quad 
+ \frac{2}{N_B} \sum_{t \in G_l} \mathbf{e}^T(t) \mathbf{O}^{\dagger }_i\mathbf{z}_{i,t:t+{\mu_i}-1}
\nonumber \\
& \stackrel{(b)}\leq F_{opt,\mathbf{s}} + tr\left ( \mathbf{O}^{\dagger}_i \mathbf{M}_{\mu_i} \mathbf{O}^{\dagger T}_i \right) \nonumber \\ & \quad - 2 \mathbb{E}\left (  \mathbf{v}^T_i(t) \mathbf{L}^T_i    \mathbf{O}^{\dagger}_i  \mathbf{v}_{i,t:t+{\mu_i}-1}  \right)   + \epsilon,
\end{align}}where $\mathbf{e}(t)$ in (a) is the state estimation error at time $t$ (in the presence of a $k$-adversary) when $\hat{\mathbf{x}}_{\mathbf{s}}(t)$ is used as the state estimate, and (b) follows from the block residue test \eqref{eq:block_residue_test_filtering}.
Using (a) and (b) above, for any $\epsilon > 0$ there exists a large enough $N_B$ such that:
\begin{align}
& \frac{1}{N_B} \sum_{t \in G_l} \mathbf{e}^T(t)\mathbf{e}(t)
+ \frac{2}{N_B} \sum_{t \in G_l} \mathbf{e}^T(t) \mathbf{O}^{\dagger }_i\mathbf{z}_{i,t:t+{\mu_i}-1} \label{eq:cross_term_filtering}  \\
& \leq F_{opt,\mathbf{s}}  - 2 \mathbb{E}\left (  \mathbf{v}^T_i(t) \mathbf{L}^T_i    \mathbf{O}^{\dagger}_i  \mathbf{v}_{i,t:t+{\mu_i}-1}  \right)   + \epsilon  \nonumber \\
& \; + \left | tr\left ( \mathbf{O}^{\dagger }_i \mathbf{M}_{\mu_i} \mathbf{O}^{\dagger T}_i \right) - \frac{1}{N_B} \sum_{t \in G_l} tr\left ( \mathbf{O}^{\dagger }_i\mathbf{z}_{i,t:t+{\mu_i}-1} \mathbf{z}^T_{i,t:t+{\mu_i}-1} \mathbf{O}^{\dagger T}_i\right ) \right| \nonumber \\
& \stackrel{(c)}\leq F_{opt, \mathbf{s}}   - 2 \mathbb{E}\left (  \mathbf{v}^T_i(t) \mathbf{L}^T_i    \mathbf{O}^{\dagger}_i  \mathbf{v}_{i,t:t+{\mu_i}-1}  \right)   + 2 \epsilon,
\label{eq:filtering_bound1}
\end{align}
where (c) follows w.h.p. from LLN as for different time indices in $G_l$, $tr\left ( \mathbf{O}^{\dagger }_i\mathbf{z}_{i,t:t+{\mu_i}-1} \mathbf{z}^T_{i,t:t+{\mu_i}-1} \mathbf{O}^{\dagger T}_i\right )$ corresponds to i.i.d. realizations of the same random variable. Similar to the prediction problem, it can be shown that the cross term $\frac{2}{N_B} \sum_{t \in G_l} \mathbf{e}^T(t) \mathbf{O}^{\dagger }_i\mathbf{z}_{i,t:t+{\mu_i}-1}$ in \eqref{eq:cross_term_filtering} has mean $- 2 \mathbb{E}\left (  \mathbf{v}^T_i(t) \mathbf{L}^T_i    \mathbf{O}^{\dagger}_i  \mathbf{v}_{i,t:t+{\mu_i}-1}  \right)$ and has vanishingly small variance as $N_B \rightarrow \infty$. This leads to the claimed bound \eqref{eq:ssf} on state estimation error and we describe the details below. 
\par For simplifying our calculations, we introduce the term $\tilde{\mathbf{e}}(t) = \mathbf{e}(t) + \mathbf{L}_i \mathbf{v}_i(t) $. Due to assumptions (A3) and (A4), $\tilde{\mathbf{e}}(t)$ is independent from $\mathbf{w}(t)$ and $\mathbf{v}_i(t)$, and hence independent from $\mathbf{z}_{i,t:t+{\mu_i}-1}$.
Now, the mean of the cross term $\frac{2}{N_B} \sum_{t \in G_l} \mathbf{e}^T(t) \mathbf{O}^{\dagger }_i\mathbf{z}_{i,t:t+{\mu_i}-1}$ can be computed as follows:
\begin{align}
& \frac{2}{N_B} \sum_{t \in G_l} \mathbb{E} \left( \mathbf{e}^T(t) \mathbf{O}^{\dagger }_i\mathbf{z}_{i,t:t+{\mu_i}-1} \right)  \nonumber \\
& = \frac{2}{N_B} \sum_{t \in G_l} \left( \mathbb{E} \left( \tilde{\mathbf{e}}^T(t) \mathbf{O}^{\dagger }_i\mathbf{z}_{i,t:t+{\mu_i}-1} \right)  - \mathbb{E} \left( \mathbf{v}_i^T(t) \mathbf{L}^T_i \mathbf{O}^{\dagger }_i\mathbf{z}_{i,t:t+{\mu_i}-1} \right) \right) \nonumber \\
& \stackrel{(a)} = - 2 \mathbb{E} \left( \mathbf{v}_i^T(t) \mathbf{L}^T_i \mathbf{O}^{\dagger }_i\mathbf{z}_{i,t:t+{\mu_i}-1} \right), \label{eq:zero_mean_filtering}
\end{align}
where (a) follows from the independence of $\tilde{\mathbf{e}}(t)$ from $\mathbf{z}_{i,t:t+{\mu_i}-1}$.
Also, using \eqref{eq:zero_mean_filtering} and taking the expectation in \eqref{eq:cross_term_filtering}:
\begin{align}
\mathbb{E} \left ( \frac{1}{N_B} \sum_{t \in G_l} \mathbf{e}^T(t)\mathbf{e}(t) \right ) & \leq F_{opt, \mathbf{s}}  + 2\epsilon \label{eq:expectation_bound_general_case} .
\end{align}

\par We now state the following claim which is useful in our variance calculation for the cross term
$\frac{2}{N_B} \sum_{t \in G_l} \mathbf{e}^T(t) \mathbf{O}^{\dagger }_i\mathbf{z}_{i,t:t+{\mu_i}-1}$. 
\begin{claim} \label{claim:bounded_e_tilde}
Consider a subset $i$ of $\mathbf{s}$ (subset of size $\theta$)
which satisfies the residue test \eqref{eq:block_residue_test_filtering}
in Algorithm \ref{alg:vector_filtering}.
With $\tilde{\mathbf{e}}(t) = \begin{bmatrix} \tilde{e}_1(t) \\ \vdots \\   \tilde{e}_n(t) \end{bmatrix} = \mathbf{e}(t) + \mathbf{L}_i \mathbf{v}_i(t) $, the following holds:
\begin{align}
\frac{1}{N_B} \sum_{t \in G_l}   \mathbb{E} \left(    \tilde{\mathbf{e}}^T(t) \tilde{\mathbf{e}}(t) \right) < \eta_1,
\end{align}
where $\eta_1$ is a constant. Furthermore, $\forall d \in \{1,2,\ldots n \}$, $\frac{1}{N_B} \sum_{t \in G_l}   \mathbb{E} \left(  \tilde{e}_d(t) \right) < \eta_2$ where $\eta_2$ is a constant.
\end{claim}
\begin{IEEEproof} See Appendix~\ref{sec:proof_claim_bounded_e_tilde}.
\end{IEEEproof}
\par Now, we will show that the variance of $\frac{1}{N_B} \sum_{t \in G_l} \mathbf{e}^T(t) \mathbf{O}^{\dagger}_i\mathbf{z}_{i,t:t+{\mu_i}-1} $ is vanishingly small as $N_B \rightarrow \infty$.
For any $\epsilon_1 > 0$, there exists a large enough $N_B$ such that:
\allowdisplaybreaks[4]{
\begin{align}
&\mathbb{E}  \left( \left( \frac{1}{N_B} \sum_{t \in G_l} \mathbf{e}^T(t) \mathbf{O}^{\dagger}_i\mathbf{z}_{i,t:t+{\mu_i}-1} \right)^2 \right)  \nonumber\\
& =\mathbb{E} \left( \frac{1}{N^2_B} \sum_{t \in G_l} \mathbf{e}^T(t) \mathbf{O}^{\dagger}_i\mathbf{z}_{i,t:t+{\mu_i}-1} \mathbf{e}^T(t) \mathbf{O}^{\dagger}_i\mathbf{z}_{i,t:t+{\mu_i}-1} \right) \nonumber \\ & \quad+ 
\mathbb{E} \left( \frac{2}{N^2_B} \sum_{t,t' \in G_l, \; t < t'} \mathbf{e}^T(t) \mathbf{O}^{\dagger}_i\mathbf{z}_{i,t:t+{\mu_i}-1} \mathbf{e}^T(t') \mathbf{O}^{\dagger}_i\mathbf{z}_{i,t':t'+{\mu_i}-1} \right)
\nonumber\\
& = \mathbb{E} \left( \frac{1}{N^2_B} \sum_{t \in G_l} \mathbf{e}^T(t) \mathbf{O}^{\dagger}_i\mathbf{z}_{i,t:t+{\mu_i}-1} \mathbf{e}^T(t) \mathbf{O}^{\dagger}_i\mathbf{z}_{i,t:t+{\mu_i}-1} \right)  \nonumber \\ & \quad + 
\mathbb{E} \left( \frac{2}{N^2_B} \sum_{t,t' \in G_l, \; t < t'} \mathbf{e}^T(t) \mathbf{O}^{\dagger}_i\mathbf{z}_{i,t:t+{\mu_i}-1} \tilde{\mathbf{e}}^T(t') \mathbf{O}^{\dagger}_i\mathbf{z}_{i,t':t'+{\mu_i}-1} \right) \nonumber \\
& \quad - \mathbb{E} \left( \frac{2}{N^2_B} \sum_{t,t' \in G_l, \; t < t'} \mathbf{e}^T(t) \mathbf{O}^{\dagger}_i\mathbf{z}_{i,t:t+{\mu_i}-1}  \mathbf{v}^T_i(t') \mathbf{L}^T_i  \mathbf{O}^{\dagger}_i\mathbf{z}_{i,t':t'+{\mu_i}-1} \right) \nonumber \\
& \stackrel{(a)}= \mathbb{E} \left( \frac{1}{N^2_B} \sum_{t \in G_l} \mathbf{e}^T(t) \mathbf{O}^{\dagger}_i\mathbf{z}_{i,t:t+{\mu_i}-1} \mathbf{e}^T(t) \mathbf{O}^{\dagger}_i\mathbf{z}_{i,t:t+{\mu_i}-1} \right)  \nonumber \\ & \quad + 
 \frac{2}{N^2_B} \sum_{t,t' \in G_l, \; t < t'} \mathbb{E} \left( \mathbf{e}^T(t) \mathbf{O}^{\dagger}_i\mathbf{z}_{i,t:t+{\mu_i}-1} \tilde{\mathbf{e}}^T(t') \right) \times \nonumber\\
& \quad \quad \quad \quad \quad \quad  \quad \quad  \quad   \quad 
 \mathbb{E} \left( \mathbf{O}^{\dagger}_i\mathbf{z}_{i,t':t'+{\mu_i}-1} \right) \nonumber \\
& \quad - \mathbb{E} \left( \frac{2}{N^2_B} \sum_{t,t' \in G_l, \; t < t'} \mathbf{e}^T(t) \mathbf{O}^{\dagger}_i\mathbf{z}_{i,t:t+{\mu_i}-1}  \mathbf{v}^T_i(t') \mathbf{L}^T_i  \mathbf{O}^{\dagger}_i\mathbf{z}_{i,t':t'+{\mu_i}-1} \right) \nonumber \\
& = \mathbb{E} \left( \frac{1}{N^2_B} \sum_{t \in G_l} \mathbf{e}^T(t) \mathbf{O}^{\dagger}_i\mathbf{z}_{i,t:t+{\mu_i}-1} \mathbf{e}^T(t) \mathbf{O}^{\dagger}_i\mathbf{z}_{i,t:t+{\mu_i}-1} \right) \nonumber \\ & \quad 
 - \mathbb{E} \left( \frac{2}{N^2_B} \sum_{t,t' \in G_l, \; t < t'} \mathbf{e}^T(t) \mathbf{O}^{\dagger}_i\mathbf{z}_{i,t:t+{\mu_i}-1}  \mathbf{v}^T_i(t') \mathbf{L}^T_i  \mathbf{O}^{\dagger}_i\mathbf{z}_{i,t':t'+{\mu_i}-1} \right) \nonumber \\
& = \mathbb{E} \left( \frac{1}{N^2_B} \sum_{t \in G_l} \mathbf{e}^T(t) \mathbf{O}^{\dagger}_i\mathbf{z}_{i,t:t+{\mu_i}-1} \mathbf{e}^T(t) \mathbf{O}^{\dagger}_i\mathbf{z}_{i,t:t+{\mu_i}-1} \right) 
\nonumber \\ & \quad -  \frac{2}{N^2_B} \sum_{t,t' \in G_l, \; t < t'} \mathbb{E} \left( \mathbf{e}^T(t) \mathbf{O}^{\dagger}_i\mathbf{z}_{i,t:t+{\mu_i}-1} \right)  \times \nonumber \\ & \quad \quad \quad \quad \quad \quad \quad \quad \quad  \mathbb{E} \left(  \mathbf{v}^T_i(t') \mathbf{L}^T_i  \mathbf{O}^{\dagger}_i\mathbf{z}_{i,t':t'+{\mu_i}-1} \right) \nonumber \\
& = \mathbb{E} \left( \frac{1}{N^2_B} \sum_{t \in G_l} \mathbf{e}^T(t) \mathbf{O}^{\dagger}_i\mathbf{z}_{i,t:t+{\mu_i}-1} \mathbf{e}^T(t) \mathbf{O}^{\dagger}_i\mathbf{z}_{i,t:t+{\mu_i}-1} \right) 
\nonumber \\ & \quad
+  \frac{2}{N^2_B} \sum_{t,t' \in G_l, \; t < t'} \mathbb{E} \left( \mathbf{v}^T_i(t) \mathbf{L}^T_i    \mathbf{O}^{\dagger}_i  \mathbf{v}_{i,t:t+{\mu_i}-1}\right) 
\times \nonumber \\ & \quad \quad \quad \quad \quad \quad \quad \quad 
 \mathbb{E} \left(  \mathbf{v}^T_i(t') \mathbf{L}^T_i  \mathbf{O}^{\dagger}_i\mathbf{v}_{i,t':t'+{\mu_i}-1} \right) \nonumber \\
& = \mathbb{E} \left( \frac{1}{N^2_B} \sum_{t \in G_l} \mathbf{e}^T(t) \mathbf{O}^{\dagger}_i\mathbf{z}_{i,t:t+{\mu_i}-1} \mathbf{e}^T(t) \mathbf{O}^{\dagger}_i\mathbf{z}_{i,t:t+{\mu_i}-1} \right) 
\nonumber \\ & \quad 
+  \frac{N_B (N_B-1)}{N^2_B} \left( \mathbb{E} \left( \mathbf{v}^T_i(t) \mathbf{L}^T_i    \mathbf{O}^{\dagger}_i  \mathbf{v}_{i,t:t+{\mu_i}-1}\right) \right)^2 \nonumber \\
& = \mathbb{E} \left( \frac{1}{N^2_B} \sum_{t \in G_l} \left( \mathbf{e}^T(t) \mathbf{O}^{\dagger}_i\mathbf{z}_{i,t:t+{\mu_i}-1}\right)^2 \right) 
\nonumber \\ & \quad 
+  \frac{N_B (N_B-1)}{N^2_B} \left( \mathbb{E} \left( \mathbf{v}^T_i(t) \mathbf{L}^T_i    \mathbf{O}^{\dagger}_i  \mathbf{v}_{i,t:t+{\mu_i}-1}\right) \right)^2 \nonumber \\
& = \mathbb{E} \left( \frac{1}{N^2_B} \sum_{t \in G_l} \left( \left(  \tilde{\mathbf{e}}^T(t) - \mathbf{v}_i^T(t) \mathbf{L}_i^T\right)   \mathbf{O}^{\dagger}_i\mathbf{z}_{i,t:t+{\mu_i}-1}\right)^2 \right) 
\nonumber \\ & \quad
+  \frac{N_B (N_B-1)}{N^2_B} \left( \mathbb{E} \left( \mathbf{v}^T_i(t) \mathbf{L}^T_i    \mathbf{O}^{\dagger}_i  \mathbf{v}_{i,t:t+{\mu_i}-1}\right) \right)^2 \nonumber \\
& = \mathbb{E} \left( \frac{1}{N^2_B} \sum_{t \in G_l} \left(  \tilde{\mathbf{e}}^T(t) \mathbf{O}^{\dagger}_i\mathbf{z}_{i,t:t+{\mu_i}-1} - \mathbf{v}_i^T(t) \mathbf{L}_i^T\mathbf{O}^{\dagger}_i\mathbf{z}_{i,t:t+{\mu_i}-1} \right)^2 \right) 
\nonumber \\ & \quad
+  \frac{N_B (N_B-1)}{N^2_B} \left( \mathbb{E} \left( \mathbf{v}^T_i(t) \mathbf{L}^T_i    \mathbf{O}^{\dagger}_i  \mathbf{v}_{i,t:t+{\mu_i}-1}\right) \right)^2 \nonumber \\
& = \mathbb{E} \left( \frac{1}{N^2_B} \sum_{t \in G_l}  \tilde{\mathbf{e}}^T(t) \mathbf{O}^{\dagger}_i\mathbf{z}_{i,t:t+{\mu_i}-1}  \tilde{\mathbf{e}}^T(t) \mathbf{O}^{\dagger}_i\mathbf{z}_{i,t:t+{\mu_i}-1} \right)
\nonumber \\ & \quad
+  \mathbb{E} \left( \frac{1}{N^2_B} \sum_{t \in G_l}  \mathbf{v}_i^T(t) \mathbf{L}_i^T\mathbf{O}^{\dagger}_i\mathbf{z}_{i,t:t+{\mu_i}-1}  \mathbf{v}_i^T(t) \mathbf{L}_i^T\mathbf{O}^{\dagger}_i\mathbf{z}_{i,t:t+{\mu_i}-1}  \right)
 \nonumber \\ & \quad 
- 2 \mathbb{E} \left( \frac{1}{N^2_B} \sum_{t \in G_l}   \tilde{\mathbf{e}}^T(t) \mathbf{O}^{\dagger}_i\mathbf{z}_{i,t:t+{\mu_i}-1} \mathbf{v}_i^T(t) \mathbf{L}_i^T\mathbf{O}^{\dagger}_i\mathbf{z}_{i,t:t+{\mu_i}-1}  \right)
\nonumber \\ & \quad
+  \frac{N_B (N_B-1)}{N^2_B} \left( \mathbb{E} \left( \mathbf{v}^T_i(t) \mathbf{L}^T_i    \mathbf{O}^{\dagger}_i  \mathbf{v}_{i,t:t+{\mu_i}-1}\right) \right)^2 \nonumber \\
& = \mathbb{E} \left( \frac{1}{N^2_B} \sum_{t \in G_l}  \tilde{\mathbf{e}}^T(t) \mathbf{O}^{\dagger}_i\mathbf{z}_{i,t:t+{\mu_i}-1}  \tilde{\mathbf{e}}^T(t) \mathbf{O}^{\dagger}_i\mathbf{z}_{i,t:t+{\mu_i}-1} \right)
\nonumber \\ & \quad
+  \frac{N_B}{N^2_B} \mathbb{E} \left(  \left( \mathbf{v}_i^T(t) \mathbf{L}_i^T\mathbf{O}^{\dagger}_i\mathbf{z}_{i,t:t+{\mu_i}-1} \right)^2 \right) \nonumber \\
& \quad 
- 2 \mathbb{E} \left( \frac{1}{N^2_B} \sum_{t \in G_l}   \tilde{\mathbf{e}}^T(t) \mathbf{O}^{\dagger}_i\mathbf{z}_{i,t:t+{\mu_i}-1} \mathbf{v}_i^T(t) \mathbf{L}_i^T\mathbf{O}^{\dagger}_i\mathbf{z}_{i,t:t+{\mu_i}-1}  \right)
\nonumber \\ & \quad
+  \frac{N_B (N_B-1)}{N^2_B} \left( \mathbb{E} \left( \mathbf{v}^T_i(t) \mathbf{L}^T_i    \mathbf{O}^{\dagger}_i  \mathbf{v}_{i,t:t+{\mu_i}-1}\right) \right)^2 \nonumber \\
& = \frac{1}{N^2_B} \sum_{t \in G_l}  \mathbb{E} \left(   \tilde{\mathbf{e}}^T(t) \mathbf{O}^{\dagger}_i\mathbf{z}_{i,t:t+{\mu_i}-1}  \tilde{\mathbf{e}}^T(t) \mathbf{O}^{\dagger}_i\mathbf{z}_{i,t:t+{\mu_i}-1} \right)
\nonumber \\ & \quad
+  \frac{N_B}{N^2_B} \mathbb{E} \left(  \left( \mathbf{v}_i^T(t) \mathbf{L}_i^T\mathbf{O}^{\dagger}_i\mathbf{z}_{i,t:t+{\mu_i}-1} \right)^2 \right) \nonumber \\
& \quad 
- 2 \mathbb{E} \left( \frac{1}{N^2_B} \sum_{t \in G_l}   \tilde{\mathbf{e}}^T(t) \mathbf{O}^{\dagger}_i\mathbf{z}_{i,t:t+{\mu_i}-1} \mathbf{v}_i^T(t) \mathbf{L}_i^T\mathbf{O}^{\dagger}_i\mathbf{z}_{i,t:t+{\mu_i}-1}  \right)
\nonumber \\ & \quad
+  \frac{N_B (N_B-1)}{N^2_B} \left( \mathbb{E} \left( \mathbf{v}^T_i(t) \mathbf{L}^T_i    \mathbf{O}^{\dagger}_i  \mathbf{v}_{i,t:t+{\mu_i}-1}\right) \right)^2 \nonumber \\
& \stackrel{(b)}= \frac{1}{N^2_B} \sum_{t \in G_l}  \mathbb{E} \left(   \tilde{\mathbf{e}}^T(t) \mathbf{O}^{\dagger}_i\mathbf{z}_{i,t:t+{\mu_i}-1}  \tilde{\mathbf{e}}^T(t) \mathbf{O}^{\dagger}_i\mathbf{z}_{i,t:t+{\mu_i}-1} \right)
\nonumber \\ & \quad
+  \frac{N_B}{N^2_B} \mathbb{E} \left(  \left( \mathbf{v}_i^T(t) \mathbf{L}_i^T\mathbf{O}^{\dagger}_i\mathbf{z}_{i,t:t+{\mu_i}-1} \right)^2 \right) \nonumber \\
& \quad 
-  \frac{2 \left(  \sum_{t \in G_l} \frac{  \mathbb{E} \left(  \tilde{\mathbf{e}}^T(t) \right)}{N_B}  \right)}{N_B} \mathbb{E} \left(   \mathbf{O}^{\dagger}_i\mathbf{z}_{i,t:t+{\mu_i}-1} \mathbf{v}_i^T(t) \mathbf{L}_i^T\mathbf{O}^{\dagger}_i\mathbf{z}_{i,t:t+{\mu_i}-1}  \right)
\nonumber \\ & \quad
+  \frac{N_B (N_B-1)}{N^2_B} \left( \mathbb{E} \left( \mathbf{v}^T_i(t) \mathbf{L}^T_i    \mathbf{O}^{\dagger}_i  \mathbf{v}_{i,t:t+{\mu_i}-1}\right) \right)^2 \nonumber \\
& = \frac{1}{N^2_B} \sum_{t \in G_l} tr \left(   \mathbb{E} \left(   \tilde{\mathbf{e}}^T(t) \mathbf{O}^{\dagger}_i\mathbf{z}_{i,t:t+{\mu_i}-1} 
  \left( \mathbf{O}^{\dagger}_i\mathbf{z}_{i,t:t+{\mu_i}-1} \right)^T \tilde{\mathbf{e}}(t)
\right) \right)
\nonumber \\ & \quad
+  \frac{N_B}{N^2_B} \mathbb{E} \left(  \left( \mathbf{v}_i^T(t) \mathbf{L}_i^T\mathbf{O}^{\dagger}_i\mathbf{z}_{i,t:t+{\mu_i}-1} \right)^2 \right) \nonumber \\
& \quad 
-  \frac{2 \left(  \sum_{t \in G_l} \frac{  \mathbb{E} \left(  \tilde{\mathbf{e}}^T(t) \right)}{N_B}  \right)}{N_B} \mathbb{E} \left(   \mathbf{O}^{\dagger}_i\mathbf{z}_{i,t:t+{\mu_i}-1} \mathbf{v}_i^T(t) \mathbf{L}_i^T\mathbf{O}^{\dagger}_i\mathbf{z}_{i,t:t+{\mu_i}-1}  \right)
\nonumber \\ & \quad
+  \frac{N_B (N_B-1)}{N^2_B} \left( \mathbb{E} \left( \mathbf{v}^T_i(t) \mathbf{L}^T_i    \mathbf{O}^{\dagger}_i  \mathbf{v}_{i,t:t+{\mu_i}-1}\right) \right)^2 \nonumber \\
& \stackrel{(c)}= \frac{1}{N^2_B} \times \nonumber \\ 
& \quad \;  \sum_{t \in G_l} tr \left(   
\mathbb{E} \left(    \mathbf{O}^{\dagger}_i\mathbf{z}_{i,t:t+{\mu_i}-1}   \left( \mathbf{O}^{\dagger}_i\mathbf{z}_{i,t:t+{\mu_i}-1} \right)^T  \right)
  \mathbb{E} \left(  \tilde{\mathbf{e}}(t)\tilde{\mathbf{e}}^T(t) \right)
\right)
\nonumber \\ & \quad
+  \frac{N_B}{N^2_B} \mathbb{E} \left(  \left( \mathbf{v}_i^T(t) \mathbf{L}_i^T\mathbf{O}^{\dagger}_i\mathbf{z}_{i,t:t+{\mu_i}-1} \right)^2 \right) \nonumber \\
& \quad 
-  \frac{2 \left(  \sum_{t \in G_l} \frac{  \mathbb{E} \left(  \tilde{\mathbf{e}}^T(t) \right)}{N_B}  \right) }{N_B}\mathbb{E} \left(   \mathbf{O}^{\dagger}_i\mathbf{z}_{i,t:t+{\mu_i}-1} \mathbf{v}_i^T(t) \mathbf{L}_i^T\mathbf{O}^{\dagger}_i\mathbf{z}_{i,t:t+{\mu_i}-1}  \right)
\nonumber \\ & \quad
+  \frac{N_B (N_B-1)}{N^2_B} \left( \mathbb{E} \left( \mathbf{v}^T_i(t) \mathbf{L}^T_i    \mathbf{O}^{\dagger}_i  \mathbf{v}_{i,t:t+{\mu_i}-1}\right) \right)^2 \nonumber \\
& \stackrel{(d)} \leq  \frac{1}{N^2_B} \sum_{t \in G_l}  \lambda^* tr \left( 
\mathbb{E} \left(  \tilde{\mathbf{e}}(t)\tilde{\mathbf{e}}^T(t)  \right)
\right)  \nonumber \\
& \quad +  \frac{N_B}{N^2_B} \mathbb{E} \left(  \left( \mathbf{v}_i^T(t) \mathbf{L}_i^T\mathbf{O}^{\dagger}_i\mathbf{z}_{i,t:t+{\mu_i}-1} \right)^2 \right) \nonumber \\
& \quad 
- \frac{2 \left(  \sum_{t \in G_l} \frac{  \mathbb{E} \left(  \tilde{\mathbf{e}}^T(t) \right)}{N_B}  \right)} {N_B} \mathbb{E} \left(   \mathbf{O}^{\dagger}_i\mathbf{z}_{i,t:t+{\mu_i}-1} \mathbf{v}_i^T(t) \mathbf{L}_i^T\mathbf{O}^{\dagger}_i\mathbf{z}_{i,t:t+{\mu_i}-1}  \right)
\nonumber \\ & \quad
+  \frac{N_B (N_B-1)}{N^2_B} \left( \mathbb{E} \left( \mathbf{v}^T_i(t) \mathbf{L}^T_i    \mathbf{O}^{\dagger}_i  \mathbf{v}_{i,t:t+{\mu_i}-1}\right) \right)^2 \nonumber \\
& =  \frac{1}{N^2_B} \sum_{t \in G_l}  \lambda^*  
\mathbb{E} \left(  \tilde{\mathbf{e}}^T(t) \tilde{\mathbf{e}}(t)
\right)
+  \frac{N_B}{N^2_B} \mathbb{E} \left(  \left( \mathbf{v}_i^T(t) \mathbf{L}_i^T\mathbf{O}^{\dagger}_i\mathbf{z}_{i,t:t+{\mu_i}-1} \right)^2 \right) \nonumber \\
& \quad 
- \frac{2 \left(  \sum_{t \in G_l} \frac{  \mathbb{E} \left(  \tilde{\mathbf{e}}^T(t) \right)}{N_B}  \right) } {N_B}
\mathbb{E} \left(   \mathbf{O}^{\dagger}_i\mathbf{z}_{i,t:t+{\mu_i}-1} \mathbf{v}_i^T(t) \mathbf{L}_i^T\mathbf{O}^{\dagger}_i\mathbf{z}_{i,t:t+{\mu_i}-1}  \right)
\nonumber \\ & \quad
+  \frac{N_B (N_B-1)}{N^2_B} \left( \mathbb{E} \left( \mathbf{v}^T_i(t) \mathbf{L}^T_i    \mathbf{O}^{\dagger}_i  \mathbf{v}_{i,t:t+{\mu_i}-1}\right) \right)^2 \nonumber \\
& \stackrel{(e)}\leq  \left( \mathbb{E} \left( \mathbf{v}^T_i(t) \mathbf{L}^T_i    \mathbf{O}^{\dagger}_i  \mathbf{v}_{i,t:t+{\mu_i}-1}\right) \right)^2 + \epsilon_1,
\end{align}}where (a) follows from the independence of $\mathbf{e}^T(t) \mathbf{O}^{\dagger}_i\mathbf{z}_{i,t:t+{\mu_i}-1} \tilde{\mathbf{e}}^T(t')$ from $\mathbf{z}_{i,t':t'+{\mu_i}-1}$, (b) follows from the independence of $\tilde{\mathbf{e}}(t)$ from $   \mathbf{O}^{\dagger}_i\mathbf{z}_{i,t:t+{\mu_i}-1} \mathbf{v}_i^T(t) \mathbf{L}_i^T\mathbf{O}^{\dagger}_i\mathbf{z}_{i,t:t+{\mu_i}-1} $, (c) follows from the independence of $ \tilde{\mathbf{e}}(t)$ from $   \mathbf{O}^{\dagger}_i\mathbf{z}_{i,t:t+{\mu_i}-1}$, (d) follows from Lemma~\ref{lemma:eigen_bound} (see Appendix~\ref{sec:trace_ineq}) with $\lambda^* = \lambda_{max} \left( \mathbb{E} \left( \mathbf{O}^{\dagger}_i\mathbf{z}_{i,t:t+{\mu_i}-1} \left( \mathbf{O}^{\dagger}_i\mathbf{z}_{i,t:t+{\mu_i}-1} \right)^T \right) \right)$,
and (e) follows from Claim~\ref{claim:bounded_e_tilde}.

\par The above result implies that the variance of the cross term $\frac{2}{N_B} \sum_{t \in G_l} \mathbf{e}^T(t) \mathbf{O}^{\dagger }_i\mathbf{z}_{i,t:t+{\mu_i}-1}$ is vanishingly small as $N_B \rightarrow \infty $.
As a result, using Chebyshev's inequality and \eqref{eq:filtering_bound1}, we have the following bound:
for any $\epsilon_2 > 0$ and $\delta > 0$, there exists a large enough $N_B$ such that:
\begin{align}
\mathbb{P} \left ( \frac{1}{N_B} \sum_{t \in G_l } \mathbf{e}^T(t) \mathbf{e}(t)  \leq F_{opt, \mathbf{s}} +  \epsilon_2 \right) \geq 1-\delta . \label{eq:N_B_final_bound_filtering}
\end{align}
Since $\frac{1}{N_B} \sum_{t \in G_l } \mathbf{e}^T(t) \mathbf{e}(t) \leq F_{opt, \mathbf{s}} +  \epsilon_2 \quad \forall l \in\{0,1,\ldots \mu_i-1\}$ implies $\frac{1}{N} \sum_{t \in G } \mathbf{e}^T(t) \mathbf{e}(t)  \leq F_{opt, \mathbf{s}} + \epsilon_2$, we have the required bound on $\frac{1}{N} \sum_{t \in G } \mathbf{e}^T(t) \mathbf{e}(t) $ from \eqref{eq:N_B_final_bound_filtering} as follows. 
For any $\epsilon_2 > 0$ and $\delta > 0$, there exists a large enough $N$ such that:
\begin{align}
\mathbb{P} \left ( \frac{1}{N} \sum_{t \in G } \mathbf{e}^T(t) \mathbf{e}(t)  \leq  F_{opt, \mathbf{s}} +  \epsilon_2 \right) \geq 1-\delta . \label{eq:N_final_bound_filtering}
\end{align}
This completes our performance analysis.
\subsection{Proof of Claim~\ref{claim:bounded_e_tilde}} \label{sec:proof_claim_bounded_e_tilde}
Using \eqref{eq:expectation_bound_general_case}:
\begin{align}
 &  \frac{1}{N_B} \sum_{t \in G_l}   \mathbb{E} \left(    \tilde{\mathbf{e}}^T(t) \tilde{\mathbf{e}}(t) \right)   +   \mathbb{E} \left(  \mathbf{v}^T_i(t)  \mathbf{L}^T_i(t)   \mathbf{L}_i \mathbf{v}_i(t)    \right)  \nonumber \\
& \leq  F _{opt, \mathbf{s}} + 2 \epsilon .
\end{align}
The above result implies that $\frac{1}{N_B} \sum_{t \in G_l}   \mathbb{E} \left(    \tilde{\mathbf{e}}^T(t) \tilde{\mathbf{e}}(t) \right) $ is bounded by a constant. This also implies that $\frac{1}{N_B} \sum_{t \in G_l}   \mathbb{E} \left(  \tilde{e}_d(t) \right)$ is bounded by a constant $\forall d \in \{1,2,\ldots n \}$ as shown below:
\begin{align}
& \frac{1}{N_B} \sum_{t \in G_l}  \mathbb{E} \left(  \tilde{e}_d(t) \right)  \nonumber \\
& \leq \frac{1}{N_B} \sum_{t \in G_l}  \left |  \mathbb{E} \left(  \tilde{e}_d(t) \right)   \right|   \nonumber\\
& \leq \frac{1}{N_B} \sum_{t \in G_l, \; \left |  \mathbb{E} \left(  \tilde{e}_d(t) \right)   \right|  < 1 } \left |  \mathbb{E} \left(  \tilde{e}_d(t) \right) \right|  + \frac{1}{N_B} \sum_{t \in G_l, \; \left |  \mathbb{E} \left(  \tilde{e}_d(t) \right)   \right|  \geq  1}   
\left | \mathbb{E} \left(  \tilde{e}_d(t) \right) \right | \nonumber \\
& \leq 1 + \frac{1}{N_B} \sum_{t \in G_l, \; \left |  \mathbb{E} \left(  \tilde{e}_d(t) \right)   \right|  \geq  1}   
\left | \mathbb{E} \left(  \tilde{e}_d(t) \right) \right | \nonumber\\
& \leq 1 + \frac{1}{N_B} \sum_{t \in G_l, \; \left |  \mathbb{E} \left(  \tilde{e}_d(t) \right)   \right|  \geq  1}   
\left | \mathbb{E} \left(  \tilde{e}_d(t) \right) \right |^2 \nonumber\\
& \stackrel{(a)}\leq 1 + \frac{1}{N_B} \sum_{t \in G_l, \; \left |  \mathbb{E} \left(  \tilde{e}_d(t) \right)   \right|  \geq  1}    \mathbb{E} \left(  \tilde{e}^2_d(t) \right) \nonumber\\
& \leq 1 + \frac{1}{N_B} \sum_{t \in G_l }    \mathbb{E} \left(  \tilde{e}^2_d(t) \right) \nonumber\\
& \leq 1 + \frac{1}{N_B} \sum_{t \in G_l }    \mathbb{E} \left(  \tilde{\mathbf{e}}^T(t) \tilde{\mathbf{e}}(t) \right),
\end{align}
where (a) follows from Jensen's inequality.
This completes the proof of Claim~\ref{claim:bounded_e_tilde}.
\end{document}